\documentclass[a4paper]{amsart}

\usepackage{amsmath,amssymb}
\usepackage{amsthm}
\usepackage{mathtools}
\usepackage{dsfont}
\usepackage{amsfonts}
\usepackage{color}
\usepackage{epsfig}
\usepackage{graphicx}
\usepackage{subcaption}
\usepackage{float}
\usepackage{esint}
\usepackage{comment}
\usepackage{tikz}
\usepackage{pgfplots}
\usepackage[
  hypertexnames=false 
]{hyperref}
\usepackage{nameref} 
\usepackage[ 
    nameinlink 
]{cleveref} 
\usepackage{autonum} 

\hypersetup{
    colorlinks=true,
    linkcolor=blue,
    citecolor=orange,
    filecolor=magenta,      
    urlcolor=cyan,
    pdftitle={Positivity and Asymptotic},
    pdfpagemode=FullScreen,
    }
\pgfplotsset{width=7cm,compat=1.9}

\calclayout

\newtheorem{Theorem}{Theorem}[section]
\newtheorem{Lemma}{Lemma}[section]
\newtheorem{Proposition}{Proposition}[section]
\newtheorem{Corollary}{Corollary}[section]

\newtheorem{Remark}{Remark}[section]

\renewcommand{\theequation}{\arabic{section}.\arabic{equation}}

\newcommand{\RR}{\mathbb{R}}
\newcommand{\cont}{\mathcal{C}}

\DeclareMathOperator*{\essinf}{ess\,inf}

\begin{document}

\title[Positivity and asymptotic behaviour]{Positivity and asymptotic behaviour of solutions to a generalized nonlocal fast diffusion equation}
\author{Arturo de Pablo, Fernando Quirós and Jorge Ruiz-Cases}

\address{F. Quirós.
  Departamento de Matemáticas, Universidad Autónoma de Madrid,
  28049-Madrid, Spain,
  and Instituto de Ciencias Matemáticas ICMAT (CSIC-UAM-UCM-UC3M),
  28049-Madrid, Spain.}
\email{fernando.quiros@uam.es}

\address{A. de Pablo.
Departamento de Matem\'{a}ticas, Universidad Carlos III de Madrid, 28911 Legan\'{e}s, Spain and ICMAT, Instituto de Ciencias Matem\'aticas (CSIC-UAM-UC3M-UCM), 28049 Madrid, Spain.}
\email{arturop@math.uc3m.es}

\address{J. Ruiz-Cases.
  Departamento de Matem\'aticas, Universidad Aut\'onoma de Madrid, 28049 Madrid, Spain and Instituto de Ciencias Matemáticas ICMAT (CSIC-UAM-UCM-UC3M),
  28049 Madrid, Spain}
\email{jorge.ruizc@uam.es}

\begin{abstract}
  We study the positivity and asymptotic behaviour of nonnegative solutions of a general nonlocal fast diffusion equation,
  \[\partial_t u + \mathcal{L}\varphi(u) = 0,\]
  and the interplay between these two properties. Here $\mathcal{L}$ is a stable-like operator and $\varphi$ is a singular nonlinearity.

  We start by analysing positivity by means of a weak Harnack inequality satisfied by a related elliptic (nonlocal) equation. Then we use this positivity to establish the asymptotic behaviour: under certain hypotheses on the nonlocal operator and nonlinearity, our solutions behave asymptotically as the Barenblatt solution of the standard fractional fast diffusion equation.

  The main difficulty stems from the generality of the operator, which does not allow the use of the methods that were available for the fractional Laplacian. Our results are new even in the case where $\varphi$ is a power.
\end{abstract}



\maketitle

\section{Introduction and main results}
\setcounter{equation}{0}

\subsection{Goal.}

The aim of this paper is to study the positivity and asymptotic behaviour of weak solutions to the generalized fast diffusion equation
\begin{equation} \label{Cauchy} \tag{P}
  \begin{cases}
    \partial_t u + \mathcal{L}\varphi(u) = 0,  \quad & (x,t) \in Q := \mathbb{R}^N \times (0,\infty), \\
    u(x,0) = u_0(x),                                 & x \in \mathbb{R}^N,
  \end{cases}
\end{equation}
where $\mathcal{L}$ is a stable-like operator, $\varphi$ a singular nonlinearity ($\varphi'(0) = +\infty$), and $u_0$ a nonnegative $L^1(\RR^N)$ function. The main difficulty comes from the operator, since the usual techniques to prove positivity cannot be applied here. Positivity is essential in our proof for the large-time behaviour in this singular case. Let us remark that it was not required to deal with the degenerate case \cite{dePabloQuiros2016}.

We define the operator $\mathcal{L}$, for smooth $f$, as
\begin{equation} \label{operator_def}
  \mathcal{L}f(x) = P.V. \int_{\mathbb{R}^N} (f(x)-f(y)) J(x,y)\, dy,
\end{equation}
where $P.V.$ is the Cauchy principal value, and $J$ is a measurable kernel such that
\begin{equation}\label{nucleo}
  \tag{H$_J$}
  \left\{\begin{aligned}
     & J(x,y) = J(y,x), \quad x,y \in \mathbb{R}^N, \, x \neq y,       \\
     & J(x,x +z) = J(x, x - z),  \quad x,z\in \RR^N, \, z \neq 0 \\
     & \frac{ 1 }{\Lambda |x-y|^{N+\sigma}} \leq J(x,y) \leq \frac{\Lambda}{|x-y|^{N+\sigma}}, \quad x,y \in \mathbb{R}^N, \, x \neq y, 
                                                                       \\
  \end{aligned} \right.
\end{equation}
for some constants $\sigma\in (0,2)$ and $\Lambda >0$. We ask the nonlinearity to satisfy:
\begin{gather} \label{phihypo} \tag{H$_\varphi 1$}
  \varphi \in \cont([0,\infty)) \cap \cont^1((0,\infty)), \quad \varphi(0) = 0, \quad \varphi \text{ strictly increasing}, \\
  \label{nonlinearcond} \tag{H$_\varphi 2$}
  \varphi^{1+A} \text{ concave for some $A>0$.}
\end{gather}
We can assume $\varphi$ to be defined only on $[0,\infty)$ since nonnegativity of the intial datum implies nonnegativity of the solution. Notice that the linear case, $\varphi(s)=s$, does not verify \eqref{nonlinearcond}; and if $\varphi(s)=s^m$, \eqref{nonlinearcond} holds with $0<A \leq \frac{1-m}{m}$, so $0<m<1$. In general, these restrictions force our nonlinearity to become singular at the origin. The above conditions on the operator and nonlinearity are assumed throughout the paper without further mention. Some results will require additional hypotheses.

\begin{Remark}
  We could replace conditions \eqref{phihypo} and \eqref{nonlinearcond} by
  \[\begin{gathered}
      \varphi \in \cont([0,\infty)) \cap \cont^2((0,\infty)),   \\
      \varphi(0) = 0, \quad \varphi'(s)>0,\quad  \text{for all } s>0, \\
      0 < A \leq -\frac{\varphi(s)\varphi''(s)}{(\varphi'(s))^2},
    \end{gathered}\]
  which may be more common in the literature, but are more restrictive.
\end{Remark}

\subsection{Main results.} \label{main_results}

The first result proved in this work is that solutions of our problem are positive in the whole space, unless they vanish identically. Positivity for general solutions follows easily from a comparison argument (truncating the initial datum) once we prove it for bounded solutions. The result for the latter case comes from the following weak Harnack inequality.

\begin{Theorem}[Positivity] \label{positivitytheorem}
  Let $u$ be a bounded weak solution of problem \eqref{Cauchy}. Then, for any $p \in [1, \frac{N}{(N-\sigma)_+})$, $x_0 \in \RR^N$, a.e.\@ $t>0$ and $R>0$,
  \begin{equation} \label{positivityinequality}
    \essinf_{x \in B_{R/2}(x_0)} \varphi(u(x,t)) \geq c\left(\int_{B_R(x_0)} \varphi(u(x,t))^{p} \,dx\right)^{1/p},
  \end{equation}
  where $c>0$ depends on $t, R$ and $\varphi$, as well as on $N$ and $\sigma$.
\end{Theorem}

For a given time, and taking $R$ big enough, this inequality will force nonnegative solutions to be positive, as long as the solution is not identically zero. Even though expression \eqref{positivityinequality} is true for a.e.\@ $t>0$, positivity will be true even without the almost everywhere condition for nontrivial solutions. The positivity result will be fundamental in the first part of the asymptotic behaviour proof, since it is required for a regularity argument providing compactness. 

We also include results of conservation of mass and extinction of solutions in finite time. Let us define $m_c = (N-\sigma)_+/N$. It is known that if $\varphi(s)\leq Cs^{m}$ for some $m>m_c$, then there is conservation of mass, see~\cite{dePabloQuiros2016}. We state that if $\varphi(s)\leq Cs^{m_c}$, then \Cref{masscons} also grants that mass is conserved. The proof is new (it was conjectured in \cite{dePabloQuiros2016}) and follows from a simplification of the proof performed in \cite{dePabloQuiros2012} for the fractional Laplacian. When $\varphi'(s) \geq cs^{m-1}$ for $m \in (0, m_c)$, \Cref{extinction} states that solutions become identically zero in finite time. This was expected, but we have not found a proof in the nonlocal literature unless $\mathcal{L} = (-\Delta)^{\sigma/2}$ and $\varphi(s) = s^m$, see \cite{dePabloQuiros2012}, so we include it for completeness. Both results are included in the \hyperref[conservationofmass]{Appendix}. Notice that, thanks to the positivity result, nontrivial nonnegative solutions are positive for every $t>0$ if there is conservation of mass ($m \geq m_c$), and they will be positive until a possible extinction time otherwise. 

Our next result is concerned with the long-time behaviour of solutions. We know, thanks to a smoothing effect (see \Cref{smoothing_1}), that our solutions go asymptotically to zero at least as $t^{-\alpha}$ for some $\alpha$ depending on $J$ and $\varphi$. We want to know how they do so: the \textit{intermediate asymptotics}. We will focus on the case where $\mathcal{L} \sim (-\Delta)^{\sigma/2}$ for far away interactions and $\varphi'(s) \sim s^{m-1}$ for small $s$ and some $m_c<m<1$. Under these hypotheses, our solution will behave like the fundamental solution (from now on, Barenblatt solution) with the same mass of the fractional fast diffusion equation with exponent $m$.

\begin{Theorem}[Asymptotic behaviour] \label{asymptoticbehaviour}
  Let $u$ be a bounded weak solution of \eqref{Cauchy}, with $M=\int_{\mathbb{R}^N} u_0$. Let $J$ and $\varphi$ satisfy
  \begin{align}
     & \lim_{|x-y|\rightarrow \infty} |x-y|^{N+\sigma} J(x,y) = c_1>0, \label{limitJ}                  \\
     & \lim_{s\rightarrow 0^+} s^{1-m} \varphi'(s) = c_2>0 \quad \text {for some }  m \in (m_c, 1), \label{limitphi}
  \end{align}
  then
  \[\|u(t) - B_M( t) \|_{L^1(\RR^N)} \rightarrow 0 \text{ as }t\rightarrow \infty, \qquad t^\alpha \|u( t) - B_M( t) \|_{L^\infty(B_{R t^{\alpha/N}})} \rightarrow 0 \text{ as }t\rightarrow \infty,\]
  where $\alpha = \frac{N}{N(m-1) + \sigma}$ and $B_M$ is the unique weak solution of
  \begin{equation} \label{Barenblatteq}
    \left\{\begin{aligned}
       & \partial_t B_M + \kappa(-\Delta)^{\sigma/2}(B_M^m) = 0 &  & \text{ in } Q,            \\
       & B_M( 0) = M\delta_0                                    &  & \text{ in } \mathbb{R}^N,
    \end{aligned}\right.
  \end{equation}
  for $\delta_0$ the unit Dirac mass placed at the origin, $\kappa = \kappa (c_1,c_2,N,\sigma)$.
\end{Theorem}
\begin{Remark}
  \emph{(i)} The boundedness condition is no restriction if $\varphi'(s)\geq c s^{m-1}$ for $s>0$, since solutions become instantaneously bounded thanks to a smoothing effect, see \Cref{smoothing_1}. 

  \noindent\emph{(ii)}  Observe that conditions \eqref{limitJ} and \eqref{limitphi} imply that only long-range interactions (the tail of the kernel) and the behaviour of the nonlinearity near the origin matter. The intuitive reasoning behind this is that mass is moving far away and solutions are going to zero.

  \noindent\emph{(iii)} We expect the $L^\infty$ result to be true in the whole space (see \cite{VazquezBarenblattnolocal}), but we have not found a way to extend the proof out of these growing domains. Useful tools like uniform Hölder estimates (see the degenerate case~\cite{dePabloQuiros2016}) and local regularizing effects (see \cite{VazquezBarenblattlocal}) would be enough to complete the proof, but are unavailable to us at the moment.

  \noindent\emph{(iv)} Observe that we do not ask the kernel $J$ of operator $\mathcal{L}$ to be of convolution type, in contrast to the asymptotic behaviour result for the degenerate case in~\cite{dePabloQuiros2016}. However, a careful inspection of the proof there shows that it can be adapted without much effort to cope with the general operators we use here.
\end{Remark}

\subsection{Outline of the proofs.}

For the positivity result, we show that bounded weak solutions of our problem are supersolutions to a linear nonlocal elliptic equation, using a result from \cite{CrandallMichaelPierre1982} that gives a bound for $\partial_t u$. Then, we adapt ideas from \cite{Kassmann2009} and \cite{KassmannFelsinger2013} to show that these supersolutions satisfy weak Harnack estimates. The technique described here also works in the local case.

On the other hand, our proof only works for singular diffusion equations, because it uses estimates which are only true for singular nonlinearities. These estimates cannot be adapted to prove the positivity for a degenerate nonlinearity. Besides, if an analogous method worked for a nonlocal degenerate diffusion equation, it would also work for the local case, which is clearly impossible because of the finite speed of propagation property, see for instance \cite{VazquezPorousMedium}. The degenerate diffusion case for stable-like operators is still open, and it should be proved in an inherently nonlocal way.

For the asymptotic behaviour result, we define rescalings of our solution and show that they converge to a solution $v$ of the standard fractional fast diffusion equation. Then we have to prove that this $v$ has $M\delta_0$ as initial condition. The main difficulty here is that the rescaled functions are not solutions of the original equation, but of some rescaled version of it. Thus, we cannot exploit the comparison principle, which is usually employed in this setting.

\subsection{Precedents.}

The basic theory of weak solutions for problem \eqref{Cauchy}, including existence, uniqueness and comparison principles, is contained in \cite{dePabloQuiros2012} when $\mathcal{L}=(-\Delta)^{\sigma/2}$ and $\varphi(s)= s^m$, and in \cite{dePabloQuiros2016} for the general case.

For the local slow diffusion equation, $\partial_t u = \Delta u^m$ with $m>1$, solutions have finite speed of propagation: if we start with a nonnegative compactly supported initial condition, the support of the solution stays compact for any later time. This is in contrast with the heat ($m=1$) and fast diffusion ($0<m<1$) equations, where the solution becomes instantaneously positive for any nonnegative initial condition. For more information on the local theory, see \cite{VazquezPorousMedium} and \cite{VazquezFastDiffusion}.

In the nonlocal case with $\mathcal{L} = (-\Delta)^{\sigma/2}$, it is already known that, both in the slow and fast diffusion cases, nonnegative solutions are positive. A proof of qualitative positivity can be found in \cite{dePabloQuiros2012}, while quantitative estimates are obtained in~\cite{BonforteVazquezQuantitativeEstimates}. Both methods differ from what we present here. In particular, even though quantitative estimates would be desired, the result found in~\cite{BonforteVazquezQuantitativeEstimates} requires the use of a type of Aleksandrov-Serrin reflection principle, which is not expected to hold for general operators such as ours without the necessary translational and rotational symmetries.

Harnack inequalities have always been a standard tool in PDEs, and have been thoroughly studied in the nonlocal setting too. We mainly rely on results and ideas from \cite{Kassmann2009} and \cite{KassmannFelsinger2013}. The novelty here is passing to the elliptic equation and then using the Harnack's inequality there to prove positivity. The original method followed here comes from Moser \cite{MoserEllipticHarnack}, see also \cite{GilbargTrudingerbook}. The idea of using estimates for $\partial_t u$ and Moser iterations in order to extend positivity has been done before in the context of local parabolic equations, see \cite{DiBenedettoHarnackBook} and, for a more recent example, \cite{LoeperQuiros2023}.

Regarding conservation of mass, the case where $m>m_c$ was proved in \cite{dePabloQuiros2016}, with the case $m = m_c$ left as a conjecture. Extinction of solutions in finite time was stated for $\mathcal{L} = (-\Delta)^{\sigma/2}$ and $\varphi(s) = s^m$ in \cite{dePabloQuiros2012} without proof, since it follows easily from the local case, see \cite{BenilanCrandallContinuousDependence}. Here we follow the same scheme of \cite{BenilanCrandallContinuousDependence}, with the necessary modifications.

An intermediate asymptotics result was already known for a general nonlocal degenerate equation, as shown in \cite{dePabloQuiros2016}. The proof presented there does not work for the fast diffusion case, and while an adaptation of that method can be attempted, here we present an alternative that should work for both slow and fast diffusion equations. The proof follows the scaling plus compactness method scheme presented for the Laplacian \cite{VazquezBarenblattlocal} and fractional Laplacian \cite{VazquezBarenblattnolocal}, but with the already mentioned added difficulty that both the operator and the nonlinearity change with the scalings. Instead we will mostly take advantage of the uniform control of the tails of our solutions. This kind of result can also be seen in a comment made in \cite{VazquezBarenblattlocal}, where it is mentioned as a curiosity.

\section{Preliminaries}
\setcounter{equation}{0}

In this section we establish several definitions and concepts that will be useful all throughout the paper.

\subsection{On the operator.}

Definition \eqref{operator_def} of the operator $\mathcal{L}$ may not make sense without the conditions~\eqref{nucleo} on the kernel, even for smooth, compactly supported functions. The second restriction of \eqref{nucleo}, $J(x,x+y) = J(x, x-y)$, is not necessary whenever $\sigma \in (0,1)$, and thus \eqref{operator_def} makes sense for $f \in \cont_c^{\sigma+\varepsilon}(\RR^N)$. Notice that, using that symmetry the operator has a pointwise expression, without the \textit{principal value}, in terms of second differences,
\begin{equation} \label{second_differences}
  \mathcal{L}f(x) =  \int_{\RR^N} \left(  f(x) - \frac{f(x+z) + f(x-z)}{2} \right) J(x,x-z)\, dz.
\end{equation}

In order to define the action of the operator $\mathcal{L}$ in a weak sense, we consider the bilinear form (nonlocal interaction energy)
\begin{equation}\label{energy_def}
  \mathcal{E}(f,g) = \frac{1}{2} \int_{\RR^N}\int_{\RR^N} (f(x)-f(y))(g(x)-g(y)) J(x,y) \, dxdy,
\end{equation}
and the quadratic form $\overline{\mathcal{E}}(f) = \mathcal{E}(f,f)$. For $f,g \in \cont^2_c(\RR^N)$ we have
\[\langle \mathcal{L}f , g \rangle = \mathcal{E}(f,g),\]
where $\langle \cdot, \cdot \rangle$ is the standard $L^2(\mathbb{R}^N)$ scalar product. As is expected, less regularity is required for the bilinear form to make sense compared to the operator. Notice that the second condition of \eqref{nucleo} is not necessary for the bilinear form to make sense. Nevertheless, it is needed in order to define the equation in a very weak sense, see below.

If $J(x,y) = \mu_{N,\sigma} |x-y|^{-(N+\sigma)}$, with the normalization constant
\[\mu_{N,\sigma} = \dfrac{2^{\sigma-1}\sigma\Gamma((N+\sigma)/2)}{\pi^{N/2}\Gamma(1-\sigma/2)},\]
the operator $\mathcal{L}$ coincides with the well known \textit{fractional Laplacian}, $\mathcal{L} = (-\Delta)^{\sigma/2}$. The bounds imposed in~\eqref{nucleo}, imply that the energy space needed to deal with weak solutions coincides with the standard \textit{fractional homogeneous Sobolev space} $\dot{H}^{\sigma/2}(\RR^N)$. We denote
\[\begin{aligned}
    [u]_{\dot{H}^{\sigma/2}(\RR^N)} & = \left(\,\iint\limits_{\RR^N \times \RR^N} \frac{(u(x)-u(y))^2}{|x-y|^{N+\sigma}}\, dxdy\right)^{1/2}, \\
    [u]_{\text{BMO}(\RR^N)}         & =  \sup_{\substack{x_0 \in \RR^N                                                                        \\ r>0 }} \frac{1}{|B_r|} \int\limits_{B_r(x_0)} \left| u(x) - \int_{B_r(x_0)} u(y) dy  \right| \, dx, \\
    [u]_{\cont^{0,\alpha}(\RR^N)}   & = \sup_{\substack{x,y\in \RR^N                                                                          \\ x\neq y}}  \frac{|u(x)-u(y)|}{|x-y|^\alpha}.
  \end{aligned}\]
We have the inequalities
\begin{equation} \label{sobolevinequalities}
  \begin{aligned}
    \|u\|_{ L^{\frac{2N}{N-\sigma} }(\RR^N)} \leq C [u]_{\dot{H}^{\sigma/2}(\RR^N)} \quad \text{ if } N > \sigma,  \\
    [u]_{\text{BMO}(\RR^N)}\leq C [u]_{\dot{H}^{\sigma/2}(\RR^N)} \quad \text{ if } N = \sigma,                    \\
    [u]_{\cont^{0, \frac{\sigma-N}{2}}(\RR^N)}\leq C [u]_{\dot{H}^{\sigma/2}(\RR^N)} \quad \text{ if } N < \sigma, \\
  \end{aligned}
\end{equation}
where the constants depend only on $N$. For $N>\sigma$, elements of $\dot{H}^{\sigma/2}(\RR^N)$ are well defined $L^{\frac{2N}{N-\sigma}}$ functions, but for $N\leq \sigma$ they are functions defined modulo a constant. This ambiguity may pose a problem when trying to change from the weak to the very weak formulation. However, it works without any problem under additional integrability conditions that $u$ will verify thanks to being a solution. Observe that $N\leq \sigma$ can only happen when $N=1$.
For more information about fractional homogeneous Sobolev spaces see \cite{FractionalHomogeneousBrascoGomezVazquez}.

\subsection{On the concept of solution.}

A \textit{weak solution} to the Cauchy problem \eqref{Cauchy} is a function $u \in \cont([0,\infty) : L^1(\RR^N))$ with $\varphi(u) \in L^2_\text{loc}((0, \infty) : \dot{H}^{\sigma/2}(\RR^N))$ such that
\begin{equation}\label{weakuation}
  \int_0^\infty \int_{\RR^N} u \partial_t \zeta \, dxdt - \int_0^\infty \mathcal{E}(\varphi(u), \zeta) \,dt = 0
\end{equation}
for every $\zeta \in \cont^\infty_c(Q)$ and taking the intial datum $u(\cdot, 0) = u_0$ almost everywhere. This equation has a comparison principle. In particular, $0 \leq u_0  \leq \mathcal{M}$ implies $0\leq u(t)\leq \mathcal{M}$ a.e.\@ in $\RR^N$ for all $t>0$. Another concept of solution is that of a \textit{very weak solution}, which is a function $u \in \cont([0,\infty) : L^1(\RR^N))$ with $\varphi(u) \in L^2_\text{loc}((0, \infty) : L^1(\RR^N, \rho dx))$, where $\rho(x) = (1+|x|^2)^{-(N+\sigma)/2}$, such that
\begin{equation}\label{veryweakuation}
  \int_0^\infty \int_{\RR^N} u \partial_t \zeta \, dxdt - \int_0^\infty \varphi(u) \mathcal{L}\zeta \,dt = 0,
\end{equation}
for every test $\zeta$. We will only use the concept of very weak solution as an intermediate step, so we will not pay much attention to the above weighted space because bounded solutions trivially belong to it. We also derive the following useful expression,
\begin{equation}\label{salto_tiempo}
  \int_{\mathbb{R}^N} (u(t+h)-u(t)) \phi(x) = - \int_t^{t+h} \mathcal{E}(\varphi(u(s)), \phi)\, ds,
\end{equation}
from taking a test function in the form $\zeta(x,s) = \phi(x) \mathds{1}_{[t,t+h]}(s)$ in \eqref{weakuation}, with $\phi \in \cont^\infty_c(\RR^N)$ and $\mathds{1}_{[t,t+h]}$ the characteristic function of the interval $[t,t+h]$ (in fact, take a suitable approximation of $\mathds{1}_{[t,t+h]}$ and pass to the limit). It is used in several calculations throughout the paper.

\subsection{On the smoothing effect.}
Regarding boundedness of solutions, under the condition that $\varphi'(s) \geq cs^{m-1}$ for $s \in (0,\infty)$, we get that solutions satisfy an $L^1$--$L^\infty$ \textit{smoothing effect}.

\begin{Theorem}[Smoothing effect]\label{smoothing_1}
  Let $0<\sigma<2$ and $\varphi'(s) \geq cs^{m-1}$ for $m>m_c = (N-\sigma)_+/N$ and some constant $c>0$. Then for any $u_0 \in L^{1}(\RR^{N})$, the weak solution of \eqref{Cauchy} verifies
  \begin{equation} \label{smoothingeq}
    \|u(t)\|_\infty \leq C t^{-\alpha}\|u_0\|_{1}^{\gamma}, \quad \text{ for all } t>0,
  \end{equation}
  with $\alpha = N /(N(m-1)+\sigma)$ and $\gamma = \sigma \alpha / N$, where the constant $C$ depends only on $m, N$, and $\sigma$.
\end{Theorem}

\begin{Remark}
  \emph{(i)} An analogous estimate is true even if $m\leq m_c$, assuming further integrability on the initial value.

  \noindent\emph{(ii)} The proof is a simple modification of the one performed in \cite{dePabloQuiros2012} and \cite{dePabloQuiros2017} for the fractional Laplacian, as stated in \cite{dePabloQuiros2016}.

  \noindent\emph{(iii)} Even though in \cite{dePabloQuiros2016} the condition that $c|s|^{m-1} \leq \varphi'(s) \leq C|s|^{m-1}$ is required when $N\leq \sigma$, only $\varphi'(s) \geq c|s|^{m-1}$ is actually needed in the proof.
\end{Remark}

\subsection{On the fundamental solution.}

Recall the Barenblatt solution $B_M$, the unique weak solution to \eqref{Barenblatteq}. It has a self-similar structure, namely
\begin{equation}
  B_M(x,t)= t^{-\alpha} B_M(xt^{-\alpha/N}, 1), \quad \alpha = \frac{N}{N(m-1) + \sigma},
\end{equation}
where $B_M(\cdot, 1)$ is a continuous bounded profile, see \cite{VazquezBarenblattnolocal}. The constant $\alpha$ appears in the smoothing effect~\eqref{smoothingeq}, and is positive if and only if $m>m_c$. The initial datum $M\delta_0$ of $B_M$ is attained as a trace in the sense of Radon measures, that is,
\begin{equation} \label{traces}
  \lim_{t\to 0^+}\int_{\RR^N} B_M(x,t) \phi(x) \, dx = \int_{\RR^N}M \phi(x) \, d\delta_0 (x) = M\phi(0),
\end{equation}
for any $\phi \in \cont_c(\RR^N)$. This is standard when considering measures as initial data. It is important to remark that Barenblatt solutions do not exist when $m\leq m_c$. More information about these Barenblatt solutions in the nonlocal case can be found in \cite{VazquezBarenblattnolocal} and in the local case in \cite{VazquezBarenblattlocal}.

\subsection{On the regularity of solutions.}

Bounded weak solutions to problem \eqref{Cauchy} are continuous, under additional assumptions on the nonlinearity $\varphi$. Indeed, if $\beta = \varphi^{-1}$ verifies
\begin{equation}
  \beta'(s) \geq cs^{\frac{1}{m}-1} \quad \text{for all } 0< s  \leq \|\varphi(u)\|_\infty \text{ and some }m \in (0,1) ,
\end{equation}
then the solution is continuous, see \cite{dePabloQuiros2018}. This condition can be interpreted as not allowing the nonlinearity to be too singular at the origin. For Hölder continuity, further conditions need to be asked to control the oscillation of the nonlinearity near the origin. In our case, we are concerned with regularity of positive solutions, so this last control is not required. It is important to notice that the regularity estimates refer to $\varphi(u)$, not the solution $u$ directly.

For the asymptotic behaviour result, \Cref{asymptoticbehaviour}, we will need to check that every function of a certain sequence satisfies the same Hölder estimates so that the sequence is equicontinuous. Hence, we will need to modify slightly the proof in \cite{dePabloQuiros2018} so that it is robust enough for this sequence.

\section{Positivity of solutions}
\setcounter{equation}{0}

To get positivity of solutions, we prove first, with the aid of a result in \cite{CrandallMichaelPierre1982}, that $\varphi(u(t))$ is a supersolution to a linear nonlocal elliptic equation, once time is frozen. Then, we adapt the weak Harnack estimates for supersolutions from \cite{Kassmann2009} and \cite{KassmannFelsinger2013}. Care has to be taken in the case where $N\leq \sigma$, since there we cannot use the standard Hardy-Littlewood-Sobolev inequality. The method described here should be easily extrapolated to other similar equations, such as the standard local fast diffusion equation. As mentioned in the introduction, it is enough to consider bounded solutions.

\subsection{From nonlinear parabolic to linear elliptic.}

This subsection relies on a result by Crandall and Pierre \cite{CrandallMichaelPierre1982} for \textit{mild solutions} to problem \eqref{Cauchy}. These are solutions in the sense of nonlinear semigroup theory, $u(t) = e^{-t\mathcal{L}\varphi} u_0$. Let us note that weak solutions are also mild solutions, which is proved in \cite{dePabloQuiros2011}, \cite{dePabloQuiros2012} and \cite{dePabloQuiros2016}.

To be able to use the necessary result, \cite[Theorem 4]{CrandallMichaelPierre1982}, we extend the definition of the nonlocal operator $\mathcal{L}$ so that it is a densely defined linear operator $\mathcal{L}:D(\mathcal{L}) \subset L^1(\mathbb{R}^N) \rightarrow L^1(\mathbb{R}^N)$ which satisfies
\begin{equation} \label{CP} 
  \begin{aligned}
     & \mathcal{L}\text{ is }m\text{-accretive in }L^1(\mathbb{R}^N) \text{ and } a\leq (I+\lambda \mathcal{L})^{-1}f \leq b \,\text{ a.e.} \\
     & \text{ for } \lambda>0, f\in L^1(\mathbb{R}^N),\text{ }a,b\in \mathbb{R}\text{ and }a\leq f \leq b \,\text{ a.e.}
  \end{aligned}
\end{equation}
A natural way to do this is to is to define the identity $\mathcal{L}u=f$ whenever $\langle u, \mathcal{L}\phi \rangle = \langle f, \phi \rangle$ for every $\phi \in \mathcal{C}^\infty_c(\mathbb{R}^N)$. The domain of the operator is
\begin{equation}
  D(\mathcal{L}) := \{u \in L^1(\mathbb{R}^N): \mathcal{L}u \in L^1(\mathbb{R}^N)\}.
\end{equation}
We know that $\mathcal{C}^\infty_c(\mathbb{R}^N) \subset D(\mathcal{L})$, as any $\phi \in \cont^\infty_c(\RR^N)$ satisfies that
\begin{equation}\label{operatoroncompact}
  |\mathcal{L}\phi(x)| \leq \frac{C_\phi}{(1+|x|^2)^{\frac{N+\sigma}{2}}} \in L^1(\mathbb{R}^N).
\end{equation}
Now, taking into account that $\mathcal{C}^\infty_c(\mathbb{R}^N)$ is dense in $L^1(\mathbb{R}^N)$, we get that $\mathcal{L}$ is densely defined in this space. The fact that our operator satisfies conditions \eqref{CP} can be checked through standard computations.

We state here the required abstract result for mild solutions. Recall the constant $A$ which is defined in condition \eqref{nonlinearcond}.

\begin{Theorem}[\cite{CrandallMichaelPierre1982}]
  If $\mathcal{L}$ satisfies \eqref{CP}, and $u$ is a mild solution to problem \eqref{Cauchy}, then
  \begin{equation}\label{nonincreasing}
    t \mapsto \frac{\varphi(u(t))}{t^{1/A}}
    \text{ is nonincreasing a.e. } x\in \mathbb{R}^N.
  \end{equation}
\end{Theorem}

This result gives us a way to change to the elliptic framework, and also tells us that if a solution is eventually positive, it had to be positive all along. This can be seen using \eqref{nonincreasing} with $t_2 \geq t_1 >0$ to get
\begin{equation}\label{previouspositivity}
  \varphi(u(t_1)) \geq \varphi(u(t_2)) \left(\frac{t_1}{t_2} \right)^{1/A}.
\end{equation}

This last fact is very particular to fast diffusion equations. The monotonicity property \eqref{nonincreasing}, formally translates into
\begin{equation}\label{formal}
  \partial_t u \leq \frac{1}{At} \left( \frac{\varphi(u)}{\varphi'(u)} \right).
\end{equation}

Since $\varphi$ is stricly increasing and concave, we can define
\begin{equation}\label{infimumito}
  \varpi = \inf_{0 < s \leq \|u\|_\infty }  \varphi'(s) = \varphi'(\|u\|_\infty)>0.
\end{equation}
Therefore, what we get if we use this condition on our equation \eqref{Cauchy} is that
\[\frac{1}{A\varpi t}\varphi(u(t)) + \mathcal{L}\varphi(u(t)) \geq 0. \]
Hence, for any fixed time $t>0$, the function $w=\varphi(u(t))$ is a supersolution of an elliptic equation. Our aim then is to prove that $w>0$ in $\RR^N$. Let us summarize all of this rigorously.

\begin{Proposition} \label{parabolictoelliptic}
  Let $u$ be a bounded weak solution of $\eqref{Cauchy}$. Then, for a.e.\@ $t>0$, the function $w=\varphi(u(t))$ is a weak supersolution of the elliptic equation
  \begin{equation}\label{elliptic}
    \eta w + \mathcal{L}w = 0,
  \end{equation}
  for $\eta = (A\varpi t)^{-1}$, where $\varpi$ is given by \eqref{infimumito}.
\end{Proposition}
\begin{proof}
  Let $t>0$ be fixed. From \eqref{nonincreasing} we deduce that, given $h>0$,
  \[\varphi(u(t+h))(t+h)^{-1/A} - \varphi(u(t))t^{-1/A} \leq 0, \quad\text{for a.e. }x\in \mathbb{R}^N.\]
  Manipulating this expression we get
  \[\big(\varphi(u(t+h))-\varphi(u(t))\big)(t+h)^{-1/A} + \varphi(u(t)) \big((t+h)^{-1/A}-t^{-1/A}\big)\leq 0.\]
  Using the Mean Value Theorem,
  \[\varphi'(c_x) (u(t+h)-u(t)) (t+h)^{-1/A}\leq - \varphi(u(t)) \big((t+h)^{-1/A}-t^{-1/A}\big), \]
  for some $c_x$ between $u(x,t)$ and $u(x,t+h)$. Rearranging the expression again we find
  \[(u(t+h)-u(t)) \leq - \frac{(t+h)^{1/A}}{\varpi} \varphi(u(t)) \big( (t+h)^{-1/A}-t^{-1/A} \big). \]
  Introduce this inequality in expression \eqref{salto_tiempo}, with $\phi\geq 0$, to get
  \[-\frac{(t+h)^{1/A}}{\varpi} \big( (t+h)^{-1/A}-t^{-1/A} \big) \int_{\mathbb{R}^N}\varphi(u(t)) \phi +  \int_t^{t+h} \mathcal{E}(\varphi(u(s)), \phi)\, ds \geq 0. \]
  We divide by $h>0$ and then take the limit when $h\rightarrow 0^+$. The right hand side may not have a limit for every $t>0$. But, since $\varphi(u) \in L^2_{\text{loc}}((0,\infty):H^{\sigma/2}(\mathbb{R}^N))$, we have that $\mathcal{E}(\varphi(u),\phi)\in L^1_{\text{loc}}((0,\infty))$. This is enough to apply Lebesgue's Differentiation Theorem for a.e. $t>0$. Choosing one of the allowed values for $t$ we obtain in the limit
  \[\frac{1}{A\varpi t}\int_{\mathbb{R}^N}\varphi(u(t)) \phi +  \mathcal{E}(\varphi(u(t)), \phi) \geq 0, \]
  as desired.
\end{proof}

\subsection{Elliptic weak Harnack inequality.}

We now prove that nonnegative bounded weak supersolutions of~\eqref{elliptic} satisfy a weak Harnack inequality
\begin{equation}\label{weakHarnack}
  \essinf_{B_{R}} w \geq c \left(\;\int\limits_{B_R}w(x)^{p} \, dx\right)^{1/p},
\end{equation}
for $c>0$ depending on $R$ and $\eta$, and for any $p \in [1, \frac{N}{(N-\sigma)_+})$. The method follows the scheme of \cite{GilbargTrudingerbook} for the local case, but we will mostly rely on the results from \cite{Kassmann2009} and \cite{KassmannFelsinger2013}, which are a nonlocal analog. We can avoid some hypotheses needed in those references since we consider our supersolutions $w$ to be nonnegative in the whole~$\RR^N$. For a solution, not just a supersolution, a complete Harnack inequality can be obtained with an extra step jumping from $p$ to $\infty$.

Let us summarize the steps of the proof. First we want to check that our supersolutions satisfy
\[\left(\;\int\limits_{B_R}w(x)^{-\overline{p}} \, dx\right)^{-1/\overline{p}} \geq c \left(\;\int\limits_{B_R}w(x)^{\overline{p}} \, dx\right)^{1/\overline{p}}, \]
for some $0<\overline{p}<1$. Afterwards, we will bound the left term by the infimum through Moser iterations
\[\essinf_{B_{R}} w \geq c \left(\;\int\limits_{B_{2R}}w(x)^{-\overline{p}} \, dx\right)^{-1/\overline{p}}.\]
Performing Moser iterations in the opposite direction will yield
\[  \left(\;\int\limits_{B_{2R}}w(x)^{\overline{p}} \, dx\right)^{1/\overline{p}} \geq c \left(\;\int\limits_{B_R}w(x)^{p} \, dx\right)^{1/p}\]
for any $p \in [1, \frac{N}{(N-\sigma)_+})$. With these inequalities, we obtain \eqref{weakHarnack} and conclude the proof. For all these results we will need to assume that $w\geq \delta>0$ on some big enough ball, but the estimates will not depend on $\delta$, so at the end we can take the limit $\delta \rightarrow 0$ and the inequalities will remain the same. 

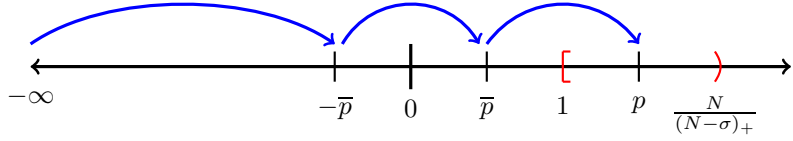
\begin{figure}[H]
  \centering
  \begin{tikzpicture}
    \def\longitudlinea{10}

    \draw[very thick, <->] (0,0) node[below = 4pt]{$-\infty$} -- (\longitudlinea,0);
    \draw[very thick] (\longitudlinea/2, 0.3) -- (\longitudlinea/2, -0.3) node[below] {$0$};
    \draw[thick] (\longitudlinea/2-1, 0.2) -- (\longitudlinea/2-1, -0.2) node[below = 2pt] {$-\overline{p}$};
    \draw[thick] (\longitudlinea/2+1, 0.2) -- (\longitudlinea/2+1, -0.2) node[below = 2pt] {$\overline{p}$};
    \draw[thick] (\longitudlinea/2+3, 0.2) -- (\longitudlinea/2+3, -0.2) node[below = 2pt] {$p$};

    \draw[thick, red] (\longitudlinea/2+2, 0.2) -- (\longitudlinea/2+2, -0.2) node[below = 2pt,black] {$1$};
    \draw[thick, red] (\longitudlinea/2+2, 0.2) -- (\longitudlinea/2+2.1, 0.2);
    \draw[thick, red] (\longitudlinea/2+2, -0.2) -- (\longitudlinea/2+2.1, -0.2);
    \draw[thick, red] (\longitudlinea/2+4, 0.2) .. controls (\longitudlinea/2+4.1, 0.05) and (\longitudlinea/2+4.1, -0.05) .. (\longitudlinea/2+4, -0.2) node[below = 2pt, black] {$\frac{N}{(N-\sigma)_+}$};

    \draw[very thick, blue, ->] (0, 0.3) .. controls (1,1) and (3,1) .. (\longitudlinea/2-1, 0.3);
    \draw[very thick, blue, ->] (4.1, 0.3) .. controls (4.5, 1) and (5.5, 1) .. (5.9, 0.3);
    \draw[very thick, blue, ->] (6, 0.3) .. controls (6.5,1) and (7.5,1) .. (8, 0.3);
  \end{tikzpicture}
  \caption{Sketch of the proof}
  \label{ejeHarnack}
\end{figure}

To simplify our computations we consider $R=1$; through a scaling argument, we will get the results for an arbitrary $R>0$. Unless it is necessary, we will also omit writing the center of the balls, that is $B_r = B_r(x_0)$, since none of the results depend on $x_0$. 

First we need the following auxiliary lemma.

\begin{Lemma}\label{prelemalog}
  Let $0<r\leq 2$ and $w$ be a nonnegative bounded weak supersolution of equation \eqref{elliptic} such that $w\geq \delta$ in $B_{3}$ for some $\delta>0$. Then,
  \begin{equation}
    \iint\limits_{B_r \times B_r} \big(\log(w(x))-\log(w(y))\big)^2 J(x,y)\, dxdy \leq C |B_{2r}| (\eta+ r^{-\sigma})
  \end{equation}
  for some constant $C$ independent of $w$, $\delta$, $r$ and $\eta$.
\end{Lemma}
\begin{proof}
  Select $\tau \in \mathcal{C}^\infty_c(B_{3r/2})$ such that $\tau \equiv 1$ in $B_r$,  $\|\tau\|_\infty \leq 1$ $\|\nabla \tau\|_\infty \leq c/\rho$. The proof follows as in \cite{Kassmann2009}, choosing $\phi=\tau^2/w$ as a test function. In our case we get
  \[-cr^{-\alpha}|B_{2r}| + \iint\limits_{B_r \times B_r} \big(\log(w(x))-\log(w(y))\big)^2 J(x,y)\, dxdy \leq \eta \int\limits_{\mathbb{R}^N}w\tau^2/w \leq \eta |B_{2r}|, \]
  and thus we obtain the result.
\end{proof}

Our next step is to use the previous result to help us prove that $\log(w)$ is a $\text{BMO}$ function. This was an idea of Moser that allows us to connect an integral with a negative power with an integral with a positive one.

\begin{Lemma}\label{loglemma}
  Let $w$ be a nonnegative bounded weak supersolution of \eqref{elliptic} such that $w\geq \delta$ in $B_{3}$ for some $\delta>0$. Then there exists $\overline{p} \in (0,1)$ such that
  \begin{equation}
    \left(\;\int\limits_{B_2}w(x)^{-\overline{p}} \, dx\right)^{-1/\overline{p}} \geq C^{-1/\overline{p}} \left(\;\int\limits_{B_2}w(x)^{\overline{p}} \, dx\right)^{1/\overline{p}},
  \end{equation}
  where $C$ is independent of $w$, $\delta$ and $\eta$, and $\overline{p}$ is independent of $w$ and $\delta$.
\end{Lemma}
In the proof we obtain the relation $\overline{p} = \dfrac{c}{\sqrt{1+\eta}}$ for some small $c>0$.

\begin{proof}
  The idea consists in proving that $\log(w) \in \text{BMO}(B_2)$ and then using an inequality of John-Niremberg's type. Choose $z_0 \in B_2$ and $r>0$ such that $B_r(z_0) \subset B_2$. We will write $B_r$ instead of $B_r(z_0)$ from now on. Using \Cref{prelemalog}, we get
  \[\iint\limits_{B_r \times B_r} \frac{\big(\log w(x)-\log w(y)\big)^2}{|x-y|^{N+\sigma}}\, dxdy \leq C
    r^N (\eta+ r^{-\sigma}). \]
  Let us write
  \[[\log w]_{B_r} = \frac{1}{|B_r|}\int\limits_{B_r} \log w(x)\, dx. \]
  Then, using Poincaré's inequality we get
  \[\int\limits_{B_r} \big| \log w(x)- [\log w]_{B_r}  \big|^2 \leq cr^\sigma\iint\limits_{B_r \times B_r} \frac{\big(\log w(x)-\log w(y)\big)^2}{|x-y|^{N+\sigma}}\, dxdy  \leq C r^{N}(1+\eta r^{\sigma}). \]
  Applying Hölder's inequality
  \[\int\limits_{B_r} \big| \log w(x)- [\log w]_{B_r} \big| \leq \left(\;\int\limits_{B_r} \big| \log w(x)- [\log w]_{B_r}  \big|^2 \right)^{1/2} |B_r|^{1/2} \leq C r^{N}\sqrt{1+\eta r^{\sigma}}. \]
  Thus, we have proved that
  \[\frac{1}{|B_r|}\int\limits_{B_r} \big| \log w(x)- [\log w]_{B_r} \big| \leq C \sqrt{1+\eta}, \]
  which implies $\log(w) \in \text{BMO}(B_2)$ and $\|\log(w)\|_{\text{BMO}(B_2)} \leq c\sqrt{1+\eta}$. This in turn yields
  \[\int_{B_2} e^{\overline{p}|\log w(x)-[\log w]_{B_2}|}\, dx \leq K, \]
  where $\overline{p} = c (1+\eta)^{-1/2}$ for some small enough constant $c>0$ only depending on the dimension $N$ (see \cite{GrafakosModern}). We assume $\overline{p}$ to be smaller than $1$. To conclude,
  \begin{equation+} \tag*{ }
    \left(\;\int\limits_{B_2} w(x)^{\overline{p}}\,dx\right)\left(\;\int\limits_{B_2} w(x)^{-\overline{p}}\, dx\right)  =  \left(\;\int\limits_{B_2} e^{\overline{p}(\log w(x)-[\log w]_{B_2})} \right) \left(\;\int\limits_{B_2} e^{-\overline{p}(\log w(x)-[\log w]_{B_2})} \right)  \leq K^2. \qedhere
  \end{equation+}
\end{proof}

The dependence of $\overline{p}$ on $\eta$ will imply, after the scaling is performed, that $\overline{p}$ depends on $R$, which did not happen in \cite{Kassmann2009}. This comes from the fact that we have the extra term $\eta w$ in the equation. If we assume $R<R_0$ for some fixed $R_0$ we can avoid this dependence.

With this last result we found the link between integrals with negative and positive powers. Next, we need an anti-Sobolev inequality, which will be required to perform the Moser iterations. But first, we need some algebraic inequalities. They will be a useful substitute to the product differentiation and chain rule used in the local case. Both are stated in \cite{KassmannFelsinger2013}, but the proof of the first one is found in \cite{Kassmann2009}.

\begin{Lemma}
  Let $a,b>0$ and $\tau_1,\tau_2\geq0$. Then,
  \begin{itemize}
    \item[i)] if $q>1$, then
          \begin{equation}\label{tauprimera}
            \begin{aligned}
              (a-b)\left(\frac{\tau_1^{q+1}}{a^q}-\frac{\tau_2^{q+1}}{b^q}\right) \leq - & \frac{\tau_1 \tau_2}{q-1}\left[ \left(\frac{a}{\tau_1}\right)^{\frac{1-q}{2}} -\left(\frac{b}{\tau_2}\right)^{\frac{1-q}{2}}  \right]^2 \\+ \max\left\{1, \frac{6q-5}{2}\right\}& (\tau_1-\tau_2)^2 \left[ \left(\frac{a}{\tau_1}\right)^{1-q} +\left(\frac{b}{\tau_2}\right)^{1-q}  \right].
            \end{aligned}
          \end{equation}
    \item[ii)]If $0<q<1$, then
          \begin{equation}\label{tausegunda}
            \begin{aligned}
              (a-b)\left(\frac{\tau_1^{2}}{a^q}-\frac{\tau_2^{2}}{b^q}\right) \leq - & \frac{2}{3(1-q)}\left( \tau_1 a^{\frac{1-q}{2}} -\tau_2b^{\frac{1-q}{2}}  \right)^2 \\+ &\frac{4q^2 - 9q + 9}{q(1-q)}(\tau_1-\tau_2)^2 \left(a^{1-q} +b^{1-q}  \right).
            \end{aligned}
          \end{equation}
  \end{itemize}
\end{Lemma}

\begin{Lemma}
  Let $w$ be a nonnegative bounded weak supersolution of \eqref{elliptic} such that $w(x) \geq \delta$ in $B_3$ for some $\delta>0$. Then, for $r\leq 2$, $\rho<1$ and $q>0$, $q\neq 1$,
  \begin{equation} \label{antiHLS}
    \iint_{B_r \times B_r} \frac{\left(w(x)^{\frac{1-q}{2}}-w(y)^{\frac{1-q}{2}}\right)^2}{|x-y|^{N+\sigma}} \,dy dx \leq C(q) (1+\eta)\rho^{-\sigma} \int_{B_{r+\rho}} w(x)^{1-q} \,dx,
  \end{equation}
  for some $C(q)>0$ satisfying $C(q) \leq cq^2$ for large $q$.
\end{Lemma}
\begin{proof}
  The main computations follow again from \cite{Kassmann2009}. Choose $\tau \in \cont^\infty_c(B_{r+\rho})$ such that $\tau \equiv 1$ on $\overline{B_r}$ and $\nabla \tau(x)\leq c\rho^{-1}$. Then, we use $\tau^{q+1} u^{-q}$ as a test function and inequality \eqref{tauprimera} if $q>1$; or $\tau^2u^{-q}$ and inequality~\eqref{tausegunda} if $0<q<1$. The main difference is that we have added to the right hand side the term
  \[ \eta \int_{\mathbb{R}^N} w(x) \tau(x)^{q+1}w(x)^{-q} \,dx  \leq \eta\int_{B_{r+\rho}} w(x)^{1-q} \,dx \leq \eta\rho^{-\sigma} \int_{B_{r+\rho}} w(x)^{1-q} \,dx. \]
  These inequalities remain the same if we use $\tau^2$ instead of $\tau^{q+1}$. Here, adding $\rho^{-\sigma}$ may seem artificial, but it helps to simplify the expression. We obtain \eqref{antiHLS}.
\end{proof}


Next we need the following auxiliary lemma, that will substitute the standard Hardy-Littlewood-Sobolev inequality when $\sigma \geq N$.

\begin{Lemma}[A General Nash-Gagliardo-Nirenberg type inequality] Let $\sigma \in (0,2)$, $B_R=B_R(x_0)$ and $u \in H^{\sigma/2}(B_R)$. Choose $\theta \in (0,1)$ such that $\theta \sigma <N$. Then, there exists a constant $c>0$ such that for any $R>0$ and any $u\in H^{\sigma/2}(B_R)$
  \begin{equation}  \label{HLS2theta}
    \begin{aligned}
      \left(\int_{B_R} |u(x)|^{\frac{2N}{N-\theta\sigma}} \, dx\right)^{\frac{N-\theta\sigma}{N}}  \leq \; & c \Bigg( R^{-\theta\sigma} \int_{B_R} |u(x)|^2\,dx                                                               \\
                                                                                                           & + \theta R^{\sigma(1-\theta)}\iint\limits_{B_R\times B_R} \frac{(u(x)-u(y))^2}{|x-y|^{N+2\sigma}}\, dxdy \Bigg).
    \end{aligned}
  \end{equation}
\end{Lemma}
\begin{proof}
  From \cite{BrezisMironescu2018} we get the interpolation inequality
  \[\|u\|_{H^{\theta\sigma/2}(B_1)} \leq C \|u\|_{L^2(B_1)}^{1-\theta} \|u\|_{H^{\sigma/2}(B_1)}^{\theta}. \]
  If $\theta \sigma<N$, we get from the Sobolev embedding
  \[\left(\int_{B_1} |u(x)|^{\frac{2N}{N-\theta\sigma}} \, dx\right)^{\frac{N-\theta\sigma}{2N}} \leq C \|u\|_{L^2(B_1)}^{1-\theta}\|u\|_{H^{\sigma/2}(B_1)}^\theta. \]
  We obtain \eqref{HLS2theta} through a scaling argument and Young's inequality.
\end{proof}

Now we prove the result that will let us reach the infimum.

\begin{Lemma}\label{iterationlemma}
  Let $w$ be a nonnegative bounded weak supersolution of \eqref{elliptic} such that $w(x) \geq \delta$ in $B_3$ for some $\delta>0$. Then, for $\overline{p} \in (0,1)$ and $p \in [1, \frac{N}{(N-\sigma)_+})$,
  \begin{align}
    \essinf_{x \in B_1} w(x)                                          & \geq C_\eta \left(\int_{B_2} w(x)^{-\overline{p}} \,dx\right)^{-1/\overline{p}}, \label{negativelimit}     \\
    \left(\int_{B_2} w(x)^{\overline{p}} \,dx\right)^{1/\overline{p}} & \geq \widetilde{C}_{\eta, \overline{p}} \left(\int_{B_1} w(x)^{p} \,dx\right)^{1/p}, \label{positivelimit}
  \end{align}
  for constants $C_\eta>0$ depending on $\eta$ and  $\widetilde{C}_{\eta, \overline{p}}>0$ depending on $\eta$ and $\overline{p}$.
\end{Lemma}
\begin{proof}
  Choose $\theta \in (0,1]$ such that $\theta\sigma < N$. Bear in mind that, if $\sigma \geq N$, we can choose $\theta$ so that $N-\theta\sigma$ is as close to zero as we want. Define $\chi = \frac{N}{N - \theta\sigma}$. We sum the integral
  \[\int_{B_r} w(x)^{1-q} \, dx \left( \leq \rho^{-\sigma}\int_{B_{r+\rho}} w(x)^{1-q} \, dx \right)\]
  on both sides of \eqref{antiHLS}. Then, we use \eqref{HLS2theta} to get
  \begin{equation} \label{baseiterations}
    \left(\int_{B_r} w(x)^{(1-q)\chi} \, dx \right)^{1/\chi} \leq C(q) (1+\eta)\rho^{-\sigma} \int_{B_{r+\rho}} w(x)^{1-q} \,dx,
  \end{equation}
  changing the constants if necessary and absorbing $c$ in $C(q)$ for simplicity. With this we have the main expression for the iterations ready. Now we have to distinguish two cases.

  \noindent \textsc{Negative iterations:} Define $p=q-1$. We focus on the range $q>1$, so $p>0$. If we raise \eqref{baseiterations} to $-p^{-1}$, and change the first ball to $B_{r-\rho}$ we get
  \[\left(\int_{B_{r-\rho}} w(x)^{-p\chi} \, dx \right)^{-1/p\chi} \geq \big(C(p) (1+\eta)\rho^{-\sigma} \big)^{-1/p} \left(\int_{B_{r}} w(x)^{-p} \,dx \right)^{-1/p}.  \]
  We need to change $\rho$ in each iteration, otherwise this will not converge. The iterations will start with $\overline{p}$ and $r=2
  $. Check that
  \[\begin{aligned}
      \left(\int_{B_{2-\rho_1-\rho_2}} w(x)^{-\overline{p}\chi^2} \, dx \right)^{-1/\overline{p}\chi^2} \geq \big(C(\overline{p}\chi) (1+\eta)\rho_2^{-\sigma} \big)^{-1/\overline{p}\chi} \left(\int_{B_{2-\rho_1}} w(x)^{-\overline{p}\chi} \,dx \right)^{-1/\overline{p}\chi} &   \\
      \geq \big[C(\overline{p})C(\overline{p}\chi)^{1/\chi}\big]^{-1/\overline{p}} \big[(1+\eta)^{1+1/\chi} \big]^{-1/\overline{p}} \big[\rho_1\rho_2^{1/\chi} \big]^{\sigma/\overline{p}} \left(\int_{B_2} w(x)^{-\overline{p}} \,dx \right)^{-1/\overline{p}}                  & .
    \end{aligned}\]
  Choose $\rho_k = 2^{-k}$. Since $\chi>1$, $p_k = \overline{p}\chi^k \rightarrow \infty$. Also notice that
  \[\sum_{k=0}^{\infty} \frac{1}{\chi^k} = \frac{1}{1-\frac{1}{\chi}} = \frac{N}{\theta\sigma}.\]
  Taking limits in the previous sequence of integrals, we get
  \[\inf_{B_1} w \geq \left(C(1+\eta)^{N/\theta \sigma} \int_{B_2} w(x)^
    {-\overline{p}} \,dx \right)^{-1/\overline{p}}, \]
  thus obtaining expression \eqref{negativelimit}.

  \noindent \textsc{Positive iterations:} Put now $p=1-q$ and consider the range $0<q<1$, so $0<p<1$. Raising both sides of \eqref{baseiterations} to $1/p$ and rearranging the expression we obtain
  \[  \left(\int_{B_{r}} w(x)^{p} \,dx \right)^{1/p}\geq \big(C(p) (1+\eta)\rho^{-\sigma} \big)^{-1/p}  \left(\int_{B_{r-\rho}} w(x)^{p\chi} \, dx \right)^{1/p\chi}. \]
  Like in the last case, we start with $\overline{p}$ and $r=2$. The main difference this time is that we can only perform a finite amount of iterations. The reason is that $p_k = \overline{p}\chi^k$ grows with $k$, so the last possible step is the one that joins $\overline{p}_{\kappa-1}<1$ to $1<\overline{p}_{\kappa}< \chi$. Thus, we can choose $\rho=1/\kappa$ for every step. This time we need
  \[\sum_{k=0}^{\kappa-1} \frac{1}{\chi^k} = \frac{1-\frac{1}{\chi^{\kappa}}}{1-\frac{1}{\chi}} = \frac{N}{\theta\sigma}\left(1- \frac{\overline{p}}{p}\right) \]
  for $p = \overline{p}\chi^{\kappa}$. Hence, the result after a finite number of steps is
  \begin{equation+} \tag*{ }
    \left(\int_{B_2} w(x)^{\overline{p}} \,dx\right)^{1/\overline{p}} \geq \left[C (1+\eta)^{\frac{N}{\theta\sigma}\left(1- \frac{\overline{p}}{p}\right) } \kappa^{\frac{N}{\theta}\left(1- \frac{\overline{p}}{p}\right)} \right]^{-1/\overline{p}} \left(\int_{B_1} w(x)^{p} \,dx\right)^{1/p}. \qedhere
  \end{equation+}
\end{proof}

Let us point out that, in \eqref{positivelimit}, making the iteration end in $B_1$ is completely arbitrary. We could have the inequality for any $B_r$ for $r<2$. Now we can join all the results and get the whole inequality for a supersolution.

\begin{Theorem} \label{superpositivity}
  Let $w$ be a bounded weak supersolution of \eqref{elliptic}. Then, for any $p \in [1, \frac{N}{(N-\sigma)_+})$,
  \begin{equation}\label{superpositivity_equation}
    \inf_{B_R} w \geq C_{\eta,R} \left(\int_{B_R} w(x)^{p} \,dx\right)^{1/p},
  \end{equation}
  where $C_{\eta, R} \sim CR^{-cR^{\sigma/2}}$ for $C>0$, $c>0$ and $R\gg 1$.
\end{Theorem}
\begin{proof}
  We assume $w \geq \delta$ on $B_3$, and take $\delta\to 0$ afterwards. Let us put together lemmas \ref{loglemma} and \ref{iterationlemma} to obtain
  \[\begin{aligned}
      \inf_{B_1} w & \geq \left[C(1+\eta)^{N/\theta \sigma} \right]^{-1/\overline{p}} \left(\int_{B_2} w(x)^
      {-\overline{p}} \,dx \right)^{-1/\overline{p}}                                                                                                                                                                                                                                                         \\
                   & \geq c^{-1/\overline{p}} \left[C(1+\eta)^{N/\theta \sigma} \right]^{-1/\overline{p}} \left(\int_{B_2} w(x)^
      {\overline{p}} \,dx \right)^{1/\overline{p}}                                                                                                                                                                                                                                                           \\
                   & \geq \left[C(1+\eta)^{N/\theta \sigma} \right]^{-1/\overline{p}} \left[C (1+\eta)^{\frac{N}{\theta\sigma}\left(1- \frac{\overline{p}}{p}\right) } \kappa^{\frac{N}{\theta}\left(1- \frac{\overline{p}}{p}\right)} \right]^{-1/\overline{p}} \left(\int_{B_1} w(x)^{p} \,dx\right)^{1/p} \\
                   & = C^{-1/\overline{p}} (1+\eta)^{\frac{N}{\theta \sigma}\left(\frac{1}{p} - \frac{2}{\overline{p}}\right)} \kappa^{\frac{N}{\theta}\left(\frac{1}{p} - \frac{2}{\overline{p}}\right)}  \left(\int_{B_1} w(x)^{p} \,dx\right)^{1/p}.
    \end{aligned}\]
  This is true for any $p \in [1, \frac{N}{(N-\sigma)_+})$ because we first fix $p$, and then choose $\overline{p} = p\chi^{-\kappa}$ for a natural number~$\kappa$ big enough such that \Cref{loglemma} can be applied.

  Since we want these estimates on balls of radius $R$, we need to perform a change of scale. In the original equation we define $\widetilde{x}=x/R$, so that we change the ball $B_R$ by the ball $B_1$. The equation \eqref{elliptic} is not invariant under this change: the constant $\eta$ must change to $\eta_R = R^\sigma \eta$. Then, all the computations we performed work and, if we undo the change of variables at the end, we get
  \[\inf_{B_R} w \geq C^{-1/\overline{p}} (1+R^\sigma\eta)^{\frac{N}{\theta \sigma}\left(\frac{1}{p} - \frac{2}{\overline{p}}\right)} \kappa^{\frac{N}{\theta}\left(\frac{1}{p} - \frac{2}{\overline{p}}\right)}  R^{-N/p}\left(\int_{B_R} w(x)^{p} \,dx\right)^{1/p}, \]
  where $\overline{p} = c/\sqrt{1 + R^\sigma \eta}$. We are particularly concerned with big enough $R$. In this case, we can consider $\overline{p} \sim R^{-\sigma/2}$, and from the formula $\overline{p}\chi^{\kappa} = p$, we get that
  $\kappa \sim \log R$. Introducing all this in the previous expression and studying the limit for large $R$, we can simplify it and obtain
  \begin{equation+} \tag*{ }
    \inf_{B_R} w \geq C R^{-c R^{\sigma/2} }\left(\int_{B_R} w(x)^{p} \,dx\right)^{1/p}. \qedhere 
    \end{equation+}
\end{proof}

Putting toghether the previous results, we easily get the proof of our positivity result.

\begin{proof}[Proof of \Cref{positivitytheorem}]
  It follows from \Cref{parabolictoelliptic} and \Cref{superpositivity}, with $w = \varphi(u(t))$.
\end{proof}

Notice that we have obtained positivity for a.e.\@ $t>0$, but using \eqref{previouspositivity} we can actually get positivity for every $t>0$.

\begin{Corollary}
  Under hypotheses of \Cref{positivitytheorem}, if moreover $c s^m \leq \varphi(s) \leq Cs^m$ with $m_c<m<1$, then for every $t>0$
  \begin{equation} \label{aespositivity}
    \essinf_{x \in B_R(x_0)} u(x,t) \geq C_{\varpi,t, R} \int_{B_R(x_0)} u(x,t) \,dx.
  \end{equation}
\end{Corollary}
\begin{proof} Just put $p = \frac{1}{m} \in [1, \frac{N}{(N-\sigma)_+})$ in \eqref{positivityinequality}. To see that it is true for every $t>0$, use \eqref{previouspositivity} and that $u \in \cont([0,\infty): L^1(\RR^N))$.
\end{proof}

The constant $C_{\varpi,t, R}$ depends on the time $t$, but we can avoid the explicit dependence if we work on compact subsets of $(0,\infty)$. For example, if $t \in (T^{-1}, T)$ for $T>1$, then we can see the constant as only depending on $T$. This will be important in the asymptotic behaviour result.

Only \eqref{positivityinequality} is required to prove positivity of the solution at $t$, since $\varphi$ is strictly increasing and $\varphi(0)=0$. In the case where $m\geq m_c$, our solutions conserve mass (see \Cref{masscons}), and thus if the initial condition is not trivially zero, we conclude positivity. If $m<m_c$, there is an extinction time for our solutions (see \Cref{extinction}), but if the initial condition is not trivially zero, we can conclude positivity for any time before the extinction time.

With \eqref{aespositivity} and the conservation of mass property, we can get a decay rate at infinity. Just take $R$ big enough.

\begin{Corollary}[Decay rate]
  Let $u$ be a bounded weak solution of \eqref{Cauchy} such that it satisfies \eqref{aespositivity} and $m\geq m_c$. Then, for $R \gg 1$, we have
  \begin{equation}\label{decayrate}
    \essinf_{x \in B_R(x_0)} u(x,t) \geq C R^{-cR^{\sigma/2}} M, \qquad M = \int_{\RR^N} u_0.
  \end{equation}
\end{Corollary}

This decay rate is far worse than the one expected, which is the one for the fractional Laplacian case, see~\cite{BonforteVazquezQuantitativeEstimates}.

\begin{Remark}
  Except for result \eqref{decayrate} and for the comments that require conservation of mass, we have not used the fact that $u(t) \in L^1(\RR^N)$. We expect the proof to work for a wider notion of weak solutions where $u(t)$ is not integrable but may satisfy some decay conditions, as long as $\varphi(u(t))$ is bounded and belongs to the fractional Sobolev space $\dot{H}^{\sigma/2}(\RR^N)$.
\end{Remark}

\section{Asymptotic behaviour}
\setcounter{equation}{0}

We study in this section the asymptotic behaviour of the solutions when the operator $\mathcal{L}$ behaves in some sense as $(-\Delta)^{\sigma/2}$, see condition \eqref{limitJ}, and the nonlinearity $\varphi$ behaves as a power in a neighbourhood of the origin, see condition \eqref{limitphi}. Given a weak solution $u$ of \eqref{Cauchy}, let us consider the sequence of functions
\begin{equation} \label{asymptoticscalings}
  u_k (x,t) = k^\alpha u(k^{\alpha/N} x, kt), \quad k\geq 1.
\end{equation}
We are going to prove that, after taking several subsequences, these scalings converge to $v$, a solution to the fractional fast diffusion equation
\begin{equation} \label{ecuacion_v}
  \partial_t v + \kappa (-\Delta)^{\sigma/2} v^m= 0. 
\end{equation}
Afterwards, we will check that $v$ takes $M\delta_0$ as an initial condition. Then, using a uniqueness result from~\cite{VazquezBarenblattnolocal}, we will identify the function $v$ as the Barenblatt solution $B_M$. Since this is true independently of the subsequence, we identify $v$ as the limit of the whole sequence $u_k$. Lastly, we will take $t=1$ and change notation so that $k=t$. One change of variables later, we conclude the result. In order to simplify notation, from now on we will consider $\kappa = 1$.

We can summarize the steps of the proof as follows: $i)$ we use the positivity result and Hölder continuity (see \cite{dePabloQuiros2018}) to get convergence on compact sets of a certain subsequence; $ii)$ uniform tail estimates give global convergence in $L^1$; $iii)$  we use weak* convergence in $L^\infty$ in combination with the previous compactness to prove that $v$ is a weak solution of \eqref{ecuacion_v}; and $iv)$ uniqueness of solutions with a Dirac delta as initial condition allows us to identify the limit, and shows that convergence is not restricted to subsequences.

For the proof of the asymptotic behaviour result, we need a control of the mass that our solutions have far away from the origin.

\begin{Proposition}[Tail control] \label{tailcontrol}
  Assume $\varphi(s) \leq C s^m$ and let $u$ be a weak solution of \eqref{Cauchy}. Then,
  \begin{equation}\label{tailestimate}
    \int_{\mathbb{R}^N\backslash B_{2R}} u(t) \leq \int_{\mathbb{R}^N\backslash B_{R}} u_0 + C \frac{t}{R^{N/\alpha}} \|u_0\|_1^m.
  \end{equation}
\end{Proposition}
\begin{proof}
  We start with the identity \eqref{salto_tiempo},
  \[\int_{\mathbb{R}^N} u(t) \phi - \int_{\mathbb{R}^N} u_0 \phi = - \int_0^t \mathcal{E}(\varphi(u), \phi) \left(= - \int_0^t \int_{\mathbb{R}^N} \varphi(u) \mathcal{L}\phi\right). \]
  The continuity of $u$ in $L^1(\mathbb{R}^N)$ at $t=0$ is vital here. What we want to control is the mass of $u$ out of a ball centered at the origin at any given time, and we are going to achieve this by controlling the mass inside the ball (since the total mass is conserved). We take $\phi = \phi_R$ for $\phi_R(x) = \psi(|x|/R)$, where $\psi$ is smooth cut-off function such that $\psi(s) = 1$ for $0\leq s \leq1$,  $\psi(s) = 0$ for $s\geq 2$, and it decreases in the interval $(1,2)$.

  It is simple to check, as in \Cref{masscons}, that $\|\mathcal{L}\phi_R\|_q \leq C R^{-\sigma + N/q}$. If we take $p=1/m$, its conjugate is $q= 1/(1-m)$. With this in mind, using Hölder's inequality,
  \[\begin{aligned}
      \left| \int_{\mathbb{R}^N} \varphi(u(t)) \mathcal{L}\phi_R\right| & \leq \|\varphi(u(t))\|_{1/m}  \|\mathcal{L}\phi_R\|_{1/(1-m)}  \leq C \|u(t)^{m}\|_{1/m} R^{-N/\alpha} \\
                                                                        & = C \|u(t)\|_1^{m} R^{-N/\alpha} = C \|u_0\|_1^{m} R^{-N/\alpha}.
    \end{aligned}\]
  Notice that
  \[\int_{\mathbb{R}^N} u(t) \phi_R - \int_{\mathbb{R}^N} u_0 \phi_R \leq \int_{B_{2R}} u(t) - \int_{B_{R}} u_0  =  \int_{\RR^N \setminus B_{R}} u_0 - \int_{\RR^N \setminus B_{2R}} u(t).\]
  Putting these last two estimates together we obtain
  \begin{equation+} \tag*{ }
    \int_{\mathbb{R}^N\backslash B_{2R}} u(t) \leq \int_{\mathbb{R}^N\backslash B_{R}} u_0 + C \frac{t}{R^{N/\alpha}} \|u_0\|_1^m. \qedhere 
  \end{equation+}
\end{proof}

This estimate by itself is interesting, since it tells us that our solution behaves as one would expect the solution of a diffusion system to behave: if the initial condition is concentrated at the origin, for small $t$ and large $R$ the solution $u$ must be close to zero. This is true even though our process is nonlocal: in the way our operator works, it allows particles to jump anywhere in space, but the probability of jumping decreases with the distance.

With this, we can now address the intermediate asymptotics result.

\begin{proof}[Proof of \Cref{asymptoticbehaviour}]
  Since our solution is bounded, we can modify the nonlinearity $\varphi(s)$ for $s>\|u\|_\infty$ and make it linear. Consequently, we may assume that there exist constants $0 < c \leq C < \infty$ such that
  \[c s^{m-1} \leq \varphi'(s) \leq C s^{m-1} \quad \text{ for all } s > 0,  \]
  since this is true for $u \approx 0$. From this we deduce the following properties for the inverse,
  \[\widetilde{c} s^{\frac{1}{m}-1} \leq \beta'(s) \leq \widetilde{C} s^{\frac{1}{m}-1} \quad \text{ for all } s > 0.\]
  Let $u_k$ be defined as in \eqref{asymptoticscalings}. It is easy to check that this scaling preserves mass. The scaled function $u_k$ satisfies
  \[\partial_t u_k + \mathcal{L}_k \varphi_k (u_k) = 0, \quad u_k(x, 0) = k^{\alpha} u_0(k^{\alpha/N} x)\]
  in a weak sense, where $\varphi_k(s)= k^{m\alpha} \varphi(s/k^\alpha)$ and $\mathcal{L}_k$ has associated kernel $J_k(x,y) = k^{\alpha(N+\sigma)/N} J(k^{\alpha/N} x, k^{\alpha/N} y)$. The family of operators $\mathcal{L}_k$ satisfies the same restrictions that $\mathcal{L}$ does, because
  \begin{equation} \label{kkernels}
    \frac{ 1 }{\Lambda |x-y|^{N+\sigma}} \leq J_k(x,y) \leq \frac{\Lambda}{|x-y|^{N+\sigma}}.
  \end{equation}
  We also observe that for $\varphi_k$ and $\beta_k = \varphi_k^{-1}$,
  \[c s^{m-1} \leq \varphi_k'(s) \leq C s^{m-1}, \quad \widetilde{c} s^{\frac{1}{m}-1} \leq  \beta_k'(s) \leq \widetilde{C} s^{\frac{1}{m}-1} \quad \text{ for all } s > 0,\]
  so the nonlinearity has the same constraints, independently of $k$. The assumptions \eqref{limitJ} and \eqref{limitphi} on the limit behaviours of $\widetilde{J}$ and $\varphi$ give
  \begin{align}
     & \lim_{k\rightarrow \infty} J_k(x, x+z) = J_\infty(z) := c_1 |z|^{-(N+\sigma)} \quad &  & \text{ uniformly for } |z| \geq K >0 ,\,x \in \RR^N,  \label{limitJ2}             \\
     & \lim_{k\rightarrow \infty} \varphi_k(s) = \varphi_\infty(s) := c_2 s^m \quad                           &  & \text{ uniformly for } 0 \leq s \leq K < \infty,  \label{limitphi2}
  \end{align}
  for an arbitrary constant $K>0$. We have now prepared all the required tools for the proof.

  \noindent \textsc{Arzelà-Ascoli compactness.}
  The smoothing effect \eqref{smoothingeq} gives
  \[\|u_k(t)\|_\infty \leq C t^{-\alpha}\|u_0\|_{1}^{\gamma} \leq C(\nu), \quad t\geq \nu >0,\]
  independently of $k\geq 1$, showing that the scalings are equibounded given a positive time. This implies
  \[\varphi'_k(u_k(t))\geq cu_k(t)^{m-1} \geq c \|u_k(t)\|_\infty^{m-1} \geq C(\nu)^{m-1} = \varpi, \]
  where $\varpi$ is independent of $k$. This, together with \eqref{aespositivity} (and the fact that all $\varphi_k(u_k)$ verify \eqref{nonincreasing}) allows us to write
  \begin{equation}\label{uniformpositivity}
    \essinf_{x \in B_R(0)} u_k(x,t) \geq C_{\varpi,t, R} \int_{B_R(0)} u_k(x,t) \,dx.
  \end{equation}
  We can ignore the dependence on $t$ of the constant as long as we work on compact subsets of $(0, \infty)$. Using the conservation of mass and the tail control \eqref{tailestimate},
  \begin{equation}
    \begin{aligned}
      \int_{B_R} u_k(x,t)\, dx & = M - \int_{\RR^N \setminus B_R} u_k(x,t)\, dx                                                      \\
                               & \geq M - \int_{\mathbb{R}^N\backslash B_{R/2}} u_k(x,0)\, dx - C \frac{t}{R^{N/\alpha}} \|u_0\|_1^m \\
                               & = M - \int_{\mathbb{R}^N\backslash B_{k^{\alpha/N} R/2}} u_0 - C \frac{t}{R^{N/\alpha}} \|u_0\|_1^m \\
                               & \geq M - \int_{\mathbb{R}^N\backslash B_{R/2}} u_0 - C \frac{t}{R^{N/\alpha}} M^m \geq \frac{M}{2},
    \end{aligned}
  \end{equation}
  where the last inequality is true for bounded time $0\leq t \leq T$, and then taking $R$ big enough appropiately. The bound is uniform in $k\geq 1$. This, together with \eqref{uniformpositivity}, implies that $u_k\geq\delta>0$ in a ball for $t\in (T^{-1},T]$ for any $T>1$. Hence, we can redefine all $\beta_k$ to be linear for small values of $s$, so that $\beta'_k(s) \geq c >0$ for $s>0$. Now, with a slight modification of the regularity proof in \cite{dePabloQuiros2018}, following the linear method in \cite{CaffarelliChanVasseur} applied to each problem with the nondegenerate $\beta_k$, we conclude that $\varphi_k(u_k)$ are Hölder continuous with uniform constants and exponents. Thus, inside the cylinder $B_n \times [\frac{1}{n}, T]$, the sequence $\{\varphi_k(u_k)\}$ is both equibounded and equicontinuous. Then we can apply Arzelà-Ascoli's Theorem and get a subsequence that converges uniformly to some function~$\widetilde{v}$. Letting $n$ get bigger, and using a diagonal argument, we may extract another subsequence (that we will denote in the same way) such that $\{\varphi_k(u_k)\}$ converges uniformly on compact sets to $\widetilde{v}$. We can also deduce that the same subsequence of $\{u_k\}$ converges uniformly to $v = \widetilde{v}^{1/m}$, since
  \[\begin{aligned}
      |u_k - v| = |\beta_k(\varphi_k(u_k)) - \beta_k(\varphi_k(v))| & \leq C |\varphi_k(u_k) - \varphi_k(v)|                                                      \\
                                                                    & \leq |\varphi_k(u_k) - \varphi_\infty(v)| +|\varphi_k(v) - \varphi_\infty(v)| \rightarrow 0
    \end{aligned}\]
  as $k \rightarrow \infty$ uniformly in $k$, due to $u_k$ and $v$ all satisfying the same uniform bounds on $B_n \times [\frac{1}{n},T]$, using that $\beta'_k$ is uniformly bounded and the uniform convergence \eqref{limitphi2}.

  \noindent \textsc{Global $L^1$ convergence.}
  Using again the tail control \eqref{tailestimate} on $u_k$, we observe that
  \begin{equation} \label{ktailcontrol}
    \int_{\mathbb{R}^N\backslash B_{2R}} u_k(t)  \leq \int_{\mathbb{R}^N\backslash B_{R}} u_k(0) + C \frac{t}{R^{N/\alpha}} \|u_0\|_1^m  = \int_{\mathbb{R}^N\backslash B_{k^{\alpha/N} R}} u_0 + C \frac{t}{R^{N/\alpha}} \|u_0\|_1^m.
  \end{equation}
  This estimate gives a lot of information about the sequence of solutions, since it is uniform in $k$ (in fact gets better as $k$ grows). In a way, we could have foreseen this: our equation may not be invariant under the scalings we performed, but we can control the key components of the equation (the kernel and the nonlinearity) by other elements that are actually invariant. Using Fatou's lemma in the expression we get
  \begin{equation}\label{vtailcontrol}
    \int_{\mathbb{R}^N\backslash B_{2R}} v(t) \leq  C \frac{t}{R^{N/\alpha}} \|u_0\|_1^m.
  \end{equation}

  With these tail control estimates we can prove that $u_k \to v$ in $L^\infty_{loc}((0,\infty):L^1(\mathbb{R}^N))$. Let us choose $\varepsilon>0$. We want to show that if $k$ is sufficiently big,
  \[\int_{\mathbb{R}^N} |u_k(x,t) - v(x,t)|\,dx < \varepsilon, \]
  where $k$ is independent of $\frac{1}{T}<t<T$, for some arbitrary $T>0$. It is trivial to see that
  \[\int_{\mathbb{R}^N} |u_k(x,t) - v(x,t)|\,dx \leq \int_{B_R} |u_k(x,t) - v(x,t)|\,dx + \int_{\mathbb{R}^N \backslash B_R} |u_k(x,t)|\,dx+ \int_{\mathbb{R}^N \backslash B_R} |v(x,t)|\,dx.\]
  Take $R$ big enough so that the second and third integrals are smaller than $\varepsilon/3$, thanks to \eqref{ktailcontrol} and \eqref{vtailcontrol}. It is simple to observe that the estimate is independent of $k$. Now that we have fixed $R$, we use that $u_k \rightarrow v$ in compact sets to bound the first integral by $\varepsilon/3$, and thus completing the proof. From this we can also conclude that $v$ has constant mass $M$ and is in $\cont((0,\infty):L^1(\mathbb{R}^N))$, since all $u_k$ verify the same. 

  \noindent \textsc{Converging to a weak solution.}
  Now let us show that $v$ is a weak solution. We will check first that it is a very weak solution, that is, the expression
  \[\int_0^\infty \int_{\mathbb{R}^N} v(x,t) \partial_t \zeta (x,t) \, dxdt + \int_0^\infty \int_{\mathbb{R}^N} v^m(x,t) (-\Delta)^{\sigma/2}\zeta(x,t)\,dxdt = 0\]
  is satisfied for every $\zeta \in \mathcal{C}^\infty_c(\mathbb{R}^N\times(0,\infty))$. We need that
  \[\begin{aligned}
      \lim_k \bigg( & \int_0^\infty \int_{\mathbb{R}^N} (v(x,t)-u_k(x,t)) \partial_t \zeta (x,t) \, dxdt                                                                        \\
      +             & \int_0^\infty \int_{\mathbb{R}^N} \left(v^m(x,t) (-\Delta)^{\sigma/2}\zeta(x,t) - \varphi_k(u_k(x,t)) \mathcal{L}_k\zeta(x,t) \right) \,dxdt  \bigg) = 0.
    \end{aligned}\]
  That the first integral goes to zero is trivial because we have uniform convergence on compact sets. The second integral requires more work. We calculate,
  \begin{equation} \label{firstsecondintegral}
    \begin{aligned}
      \int_0^\infty \int_{\mathbb{R}^N} & \Big(v^m(x,t) (-\Delta)^{\sigma/2} \zeta(x,t) - \varphi_k(u_k(x,t)) \mathcal{L}_k\zeta(x,t) \Big)\,dxdt  =             \\
                                        & \int_0^\infty \int_{\mathbb{R}^N} \Big(v^m(x,t)- \varphi_k(u_k(x,t) \Big) (-\Delta)^{\sigma/2}\zeta(x,t)\, dx dt       \\
      +                                 & \int_0^\infty \int_{\mathbb{R}^N} \varphi_k(u_k(x,t)) \Big((-\Delta)^{\sigma/2}-\mathcal{L}_k\Big)\zeta(x,t) \, dxdt .
    \end{aligned}
  \end{equation}
  For the first integral term, since the sequence $\{\varphi_k(u_k)\}$ is uniformly bounded in $L^\infty((\mathbb{R}^N)\times[\nu, \infty))$ (we take $\nu$ as small as needed so it lies outside the support of $\zeta$), we have weak* convergence of a subsequence in $L^\infty((\mathbb{R}^N)\times[\nu, \infty))$. We have uniform convergence in compact sets to $v^m$, so this weak* limit must be $v^m$. We took these extra steps to get weak convergence because it is global, and the local uniform convergence may not be enough due to $(-\Delta)^{\sigma/2}\zeta$ not having spatial compact support in general. For every $\phi \in \mathcal{C}^\infty_c(\mathbb{R}^N)$ there exists a constant $C_\phi$ such that
  \[\left|(-\Delta)^{\sigma/2}\phi (x) \right| \leq  \frac{C_\phi }{(1+|x|^2)^{\frac{N+\sigma}{2}}}\in L^1(\mathbb{R}^N).\]
  Observe that same estimate, independent of $k$, holds for $\mathcal{L}_k$ thanks to \eqref{kkernels}.
  Therefore, we obtain that the first integral term on the right-hand side of \eqref{firstsecondintegral} tends to zero with $k$. For the second integral term we use the pointwise convergence
  \[\lim_{k \to \infty} \Big((-\Delta)^{\sigma/2}-\mathcal{L}_k\Big)\zeta(x,t) = 0.\]
  and the dominated convergence theorem. The proof of this is simple using \eqref{second_differences}, the expression of the operator in terms of second differences,
  \[\begin{aligned}
      \left|\Big((-\Delta)^{\sigma/2}-\mathcal{L}_k\Big)\zeta(x,t)\right| & \leq \frac{1}{2} \int_{\RR^N} |\zeta(x+z,t) + \zeta(x-z,t) -2 \zeta(x,t)|\,|J_k(x, x+z)-J_\infty(z)| dz                     \\
                                                                          & \leq C \int_{|z|\leq r} \frac{|z|^{2}}{|z|^{N+\sigma}} dz + C \int_{|z|\geq r} |J_k(x, x+z) - J_\infty(z)|dz < \varepsilon.
    \end{aligned}\]
  This last inequality is true as long as we choose $r$ small enough so that this first integral is smaller than $\varepsilon/2$, and then $k$ large enough so that this second integral is smaller than $\varepsilon/2$ thanks to \eqref{limitJ2}. This convergence is uniform on compact intervals of time. All in all, using this convergence, that $\big((-\Delta)^{\sigma/2}-\mathcal{L}_k\big)\zeta$ belongs uniformly to $L^1(Q)$, and the fact that $\varphi_k(u_k)$ is uniformly bounded, we get that this second integral term tends to zero too. With this, we have that $v$ is a very weak solution.

  For $v$ to be actually a weak solution we need to check that it belongs to the proper energy space, $v^m \in L^2_{\text{loc}}((0,\infty); \dot{H}^{\sigma/2})$. Solutions $u_k$ verify the inequality
  \[\int_{\frac{1}{T}}^T \overline{\mathcal{E}}_k(\varphi_k(u_k(t))) \, dt \leq \left\|\varphi_k\left(u_k\left( 1/T \right) \right) \right\|_{\infty} \left\|u_k\left( 1/T \right) \right\|_{1} \leq C\quad \text{ uniformly for }k\geq 1,\]
  for any $T>1$. This, at least formally, can be quickly proved using $\varphi(u)$ as a test function; while the rigorous argument requires the standard Steklov averages. The right-hand side does not depend on $k$, taking advantage of the $L^1$ and $L^\infty$ bounds we have been using so far. Then, thanks to \eqref{kkernels},
  \[\int_{\frac{1}{T}}^T \big[\varphi_k(u_k(t))) \big]_{\dot{H}^{\sigma/2}(\RR^N)}^2 \, dt \leq C \quad \text{ uniformly for }k\geq 1. \]
  We now use Fatou's Lemma and inequalities \eqref{sobolevinequalities} to prove that $v^m \in L^2_{\text{loc}}((0,\infty): \dot{H}^{\sigma/2})$. In the case $N=\sigma$ we don't actually use the inequality \eqref{sobolevinequalities}, but instead we use equiboundedness of $\varphi_k(u_k(t))$ for $t\geq 1/T$. Hence, $v$ is a weak solution.

  \noindent \textsc{Identifying the Barenblatt solution.}
  It only remains to identify $v$ as the Barenblatt solution~$B_M$. Using \cite[Theorem 7.1]{VazquezBarenblattnolocal}, it is sufficient to check that $v$ takes $M \delta_0$ as the initial datum, that is, for any $\phi \in \mathcal{C}_c(\mathbb{R}^N)$,
  \[\lim_{t\rightarrow 0} \int_{\mathbb{R}^N} v(x,t) \phi(x) \, dx = M\phi(0).
  \]
  Let us choose $\varepsilon>0$. Thanks to uniform continuity, we may take $\delta>0$ such that $\|\phi(\cdot)- \phi(0)\|_\infty < \varepsilon/(2M)$ as long as $|x| < \delta$. Then,
  \[\begin{aligned}
      \left|\int_{\mathbb{R}^N} v(x,t) \phi(x) \, dx - M\phi(0) \right| & = \left|\int_{\mathbb{R}^N} v(x,t) \big(\phi(x)  -\phi(0) \big) \, dx \right|                                                                                 \\
                                                                        & \leq  \left|\int_{B_\delta} v(x,t) (\phi(x)  -\phi(0)) \, dx \right| +  \left|\int_{\mathbb{R}^N\backslash B_\delta} v(x,t) (\phi(x)  -\phi(0)) \, dx \right| \\
                                                                        & \leq \frac{\varepsilon}{2} +\|\phi(\cdot)- \phi(0)\|_\infty  \int_{\mathbb{R}^N\backslash B_\delta} v(x,t)                                                    \\
                                                                        & \leq \frac{\varepsilon}{2} +C M^m\|\phi(\cdot)- \phi(0)\|_\infty  \frac{t}{\delta^{N/\alpha}} < \varepsilon,
    \end{aligned}\]
  taking $t<\rho \leq \frac{\varepsilon\delta^{N/\alpha}}{2C M^m\|\phi(\cdot)- \phi(0)\|_\infty}$. Notice that, since we identified the limit of the subsequence of $\{u_k\}$ as $B_M$, and it did not depend on said subsequence, we know that the original sequence of scalings $\{u_k\}$ has $B_M$ as its limit.

  The final step is standard: we take $t=1$ and change the notation so that $k=t$. We define the new variable $\xi=xt^{\alpha/N}$ and use the self-similar structure of $B_M$ to get
  \[|t^{\alpha}u(xt^{\alpha/N}, t) - B_M(x,1)| = t^{\alpha} |u(\xi, t) - B_M(\xi,t)|.\]
  Both the $L^1$ and $L^\infty$ convergence follow from this expression.
\end{proof}

\begin{Remark}
  A global convergence result in $L^\infty$ would only say, for outer scales $|x|\gg t^{\alpha/N}$ that $u(t) = o(t^{-\alpha})$, and would give neither a sharp decay rate, nor an asymptotic profile. The behaviour for outer scales for general initial data is known to be a hard problem even for the standard local heat equation.
\end{Remark}

\appendix

\section*{Appendix: Conservation of mass and extinction}\label{conservationofmass}

\renewcommand{\theTheorem}{A.\arabic{Theorem}}
\renewcommand{\theRemark}{A.\arabic{Remark}}
\renewcommand{\theequation}{A.\arabic{equation}}

\setcounter{Theorem}{0}
\setcounter{Remark}{0}
\setcounter{equation}{0}

From now on we consider general sign-changing solutions.

We show that solutions to our equation conserve mass whenever $|\varphi(s)|\leq C |s|^{m_c}$. Recall the critical exponent $m_c= \frac{(N-\sigma)_+}{N}$. This was already known to happen if $\varphi(s)\leq C|s|^m$ for $m>m_c$, as shown in \cite{dePabloQuiros2016}, and conjectured for $m=m_c$ in the same paper. For the fractional Laplacian both cases were proved in \cite{dePabloQuiros2012}. Our proof for the case $m=m_c$ is a simplified version of the proof found in this last paper, which relied on writing $(-\Delta)^{\sigma/2} = -\Delta (-\Delta)^{-\frac{2-\sigma}{2}}$. We avoid using the Laplacian so that the result can be extended to more general nonlocal operators.

\begin{Theorem}[Conservation of mass]\label{masscons}
  Let $u$ be a bounded weak solution of \eqref{Cauchy}. If $|\varphi(s)|\leq C|s|^{m_c}$, then
  \[\int_{\mathbb{R}^N}u(t) = \int_{\mathbb{R}^N} u_0,\]
  for every $t>0$.
\end{Theorem}
\begin{proof} We will divide the proof in two cases.

  \noindent \textsc{Case 1:} $m_c>0$. Define a radially non-increasing cut-off function $\phi_1 \in \mathcal{C}^\infty_c(\mathbb{R}^N)$ such that $\phi_1 \equiv 1$ in $B_1$ and $\phi_1\equiv 0$ in $\mathbb{R}^N \backslash B_2$. We then define $\phi_R (x) = \phi_1(\frac{x}{R})$. It is easy to see that $\mathcal{L}\phi_1 \in L^1(\mathbb{R}^N)\cap L^\infty(\mathbb{R}^N)$. Then,
  \[\mathcal{L}\phi_R(x) = R^{-\sigma} \widetilde{\mathcal{L}} \phi_1 \left(\frac{x}{R} \right) \]
  for $\widetilde{\mathcal{L}}$ the operator with kernel $\widetilde{J}(x,y) = R^{N+\sigma}J(x/R,y/R)$, which satisfies the same conditions \eqref{nucleo} as $J$. Hence, $\widetilde{\mathcal{L}}\phi_1 \in L^1(\mathbb{R}^N)\cap L^\infty(\mathbb{R}^N)$, and we get that $\|\mathcal{L}\phi_R\|_q = C R^{-\sigma + N/q}$ for every $1 \leq q \leq \infty$. Let us fix $t>0$. We observe that for every $\delta>0$ we can choose $R_0$ such that
  \[\int_{\RR^N \setminus B_{R_0}} |u(t)| \leq \delta. \]
  We divide $u(t)$ as the sum of two functions with disjoint support:
  \[u(t) = u_1 + u_2, \quad u_1 = u(t) \mathds{1}_{B_{R_0}}, \quad u_2 = u(t)-u_1.\]
  The idea behind this separation is that $u_1$ is going to be a good function, and $u_2$ will not be as good but satisfies
  \[\int_{\mathbb{R}^N} |u_2| \leq \delta.\]
  Another important detail is that, due to the disjoint support of these functions and the fact that $\varphi(0) =0$, we have $\varphi(u(t)) = \varphi(u_1) + \varphi(u_2)$. Introducing this into \eqref{salto_tiempo} we get
  \[\left|\int_{\mathbb{R}^N} u(t) \phi_R- \int_{\mathbb{R}^N} u_0 \phi_R \right| \leq \int_0^t \int_{\mathbb{R}^N} |\varphi(u_1)| \, |\mathcal{L}\phi_R| + \int_0^t \int_{\mathbb{R}^N} |\varphi(u_2)| \, |\mathcal{L}\phi_R| = I_1 + I_2.\]
  In order to estimate $I_{1}$ we exploit the fact that $\varphi(u_1) \in L^1(\RR^N)$ thanks to the compact support of $u_1$, so that
  \begin{equation}
    I_1 \leq C t R^{-\sigma} \|\varphi(u_1)\|_1 \leq  C t R^{-\sigma} \|u_1^{m_c}\|_{L^1(B_{R_0})} \leq C t R^{-\sigma} R_0^{\sigma} \|u_0\|_1^{m_c} \leq C t \left(\frac{R_0}{R}\right)^\sigma \|u_0\|_1^{m_c} .
  \end{equation}
  We have used Hölder's inequality with exponents $1/m_c$ and $1/(1-m_
    c)$ for $u$ and $\mathds{1}_{B_{R_0}}$ respectively. The estimate we need for $I_2$ comes again from a Hölder's inequality,
  \[\int_0^t \int_{\mathbb{R}^N} |\varphi(u_2)| \, |\mathcal{L}\phi_R| \leq  Ct \|u_2\|_1^{m_c} \|\mathcal{L}\phi_R\|_{\frac{N}{\sigma}} \leq Ct\delta^{m_c}.\]
  All in all,
  \[ \left|\int_{\mathbb{R}^N} u(t) \phi_R- \int_{\mathbb{R}^N} u_0 \phi_R \right| \leq Ct \left[ \left(\frac{R_0}{R}\right)^\sigma \|u_0\|_1^{m_c} + \delta^{m_c} \right].\]
  Take limits when $R\rightarrow \infty$ first and then $\delta \rightarrow 0$ to finish the proof.

  \noindent \textsc{Case 2:} $m_c=0$. It follows similarly to the previous one with a few changes. The estimate for $I_1$ is altered slightly,
  \[I_1 \leq Ct \frac{R_0^N}{R^\sigma},\]
  which still goes to zero when $R\to \infty$. The estimate for $I_2$ is the one that suffers the biggest change and we have to differentiate the cases where $N<\sigma$ and $N=\sigma$.

  If $N<\sigma$, since $\|\mathcal{L}\phi_R\|_1 = CR^{N-\sigma}$,
  \[I_2 = \int_0^t \int_{\RR^N} |\varphi(u_2)|\,|\mathcal{L}\phi_R| \leq Ct \int_{\RR^N}|\mathcal{L}\phi_R| \leq Ct R^{N-\sigma},\]
  and thus, taking $R\to \infty$ finishes the argument.

  If $N=\sigma$, notice that $\|\mathcal{L}\phi_R\|_1 = C$, that is, the integral does not depend on $R$. Using that $\varphi(u_2) = \mathds{1}_{\RR^N\setminus B_{R_0}}$, we get
  \[I_2 = \int_0^t \int_{\RR^N} |\varphi(u_2)|\,|\mathcal{L}\phi_R| \leq Ct \int_{\RR^N \setminus B_{R_0}}|\mathcal{L}\phi_1|. \]
  Now, taking the limit $R \to \infty$ makes $I_1$ go to zero. Then we take the limit $R_0 \to \infty$, forcing $I_2$ to go to zero too.
\end{proof}

\begin{Remark}
  Notice that in the case $N\leq \sigma$, for any given nonlinearity $\varphi$, any bounded solution conserves mass. This is due to the restriction here being $|\varphi(s)| \leq C$ for $s\geq 0$. If the solution is bounded, we can change the nonlinearity $\varphi(s)$ for big enough values of $|s|$ so that it is bounded too, and thus satisfies the hypothesis.
\end{Remark}

Let us now show that if the nonlinearity is singular enough (meaning $\varphi'(s)\geq |s|^{m-1}$ for $0<m<m_c$), then solutions are identically zero after a certain time. We split the initial condition $u_0$ into its positive part $u_0^+$ and negative part $u_0^-$. A comparison principle tells us that if this extinction is true for solutions starting from $u_0^+$ and $u_0^-$, it is also true for the original solution. Hence, we can assume that solutions are nonnegative without loss of generality.

\begin{Theorem}[Extinction]\label{extinction}
  Let $u$ be a nonnegative weak solution of \eqref{Cauchy} and $0<m<m_c$. If $ \varphi'(s) \geq c s
      ^{m-1}$ and $u_0 \in L^1(\mathbb{R}^N) \cap L^\infty(\mathbb{R}^N)$, then there is a time $T>0$ such that $u(x,t) = 0$ a.e.\@ in $\mathbb{R}^N$ for all $t\geq T$.
\end{Theorem}

\begin{proof} A formal proof follows from mutiplying the equation by $u^{p-1}$ with $p=(1-m)N/\sigma$ and proceeding as explained in \cite{dePabloQuiros2012}. Then, for $J(t) = \|u\|_p^p$ we get
  \begin{equation}\label{extinctionODE}
    J'(t) + C J^{\frac{N-\sigma}{N}}(t) \leq 0
  \end{equation}
  for $C>0$, and from this ODE we can conclude the result. Since we want to extend the proof to be valid for a more general nonlinearity $\varphi(s)$ that is not necessarily a power, we will explain the rigorous details. Basically, we have to justify that from expression \eqref{salto_tiempo},
  choosing an appropiate test function $\phi \in \mathcal{C}^\infty_c(\mathbb{R}^N)$, we get
  \begin{equation} \label{discreto_final}
    \int_{\mathbb{R}^N} u(\tau+h)^p-\int_{\mathbb{R}^N}u(\tau)^p  + C \int_\tau^{\tau+h} \left(\int_{\mathbb{R}^N} u(t)^p \right)^{\frac{N-\sigma}{N}} \, dt  \leq 0.
  \end{equation}
  Then we get \eqref{extinctionODE} in a distributional sense and the same conclusion follows, exactly as in the local case, see~\cite{BenilanCrandallContinuousDependence}. The main difference between the formal and rigorous proof is the lack of both temporal and spatial regularity. The temporal regularity problem is solved by using the Steklov average for $\delta>0$,
  \[u^\delta(t) = \frac{1}{\delta}\int_t^{t+\delta} u(s)\, ds,\]
  and the spatial regularity is solved by using the approximation explained below.

  Notice that $u \in L_\text{loc}^2((0,\infty):\dot{H}^{\sigma/2}(\RR^N))$, since $\beta = \varphi^{-1}$ satisfies $\beta(0)=0$ and $\beta' \in \cont^1(\RR)$. If $p\geq 2$, then we can easily see that the function $u^{p-1}$ belongs to the energy space $L_\text{loc}^2((0,\infty):\dot{H}^{\sigma/2}(\RR^N))$. The condition $p\geq 2$ is equivalent to $m \leq m^{*}=\frac{N-2\sigma}{N}$. But $m^{*}<m_c$, so we could find ourselves in the situation that $m^{*} < m < m_c$, which would translate into $1<p<2$. Since we cannot be sure that $u^{p-1}$ belongs to the energy space in this case, we have to treat this with more care. In order to do so, let us define $f_\varepsilon(s) = (s+\varepsilon)^{p-1}-\varepsilon^{p-1}$.
  Now we can use $f_\varepsilon \left( u^\delta(t) \right)$ as a test function, since we have removed the singularity of the power at the origin. Introducing it in \eqref{salto_tiempo} we get
  \begin{equation} \label{extincion_intermedia}
    \int_{\mathbb{R}^N} \partial_t(u^\delta(t)) f_\varepsilon \left( u^\delta(t) \right) +
    \frac{1}{\delta}\int_t^{t+\delta}\mathcal{E} \Big(\varphi(u(s)), f_\varepsilon \left( u(s) \right) \Big) \, ds = 0.
  \end{equation}
  Notice that $\partial_t(u^\delta(t)) = \frac{u(t+\delta)-u(t)}{\delta}$. Let us estimate the first term. Using the chain rule, we get
  \begin{equation} \label{primitiva}
   \partial_t(u^\delta(t)) f_\varepsilon \left( u^\delta(t) \right) = \partial_t F_\varepsilon(u^\delta(t)), \quad \text{where } 
    F_\varepsilon(s) = \int_0^s f_\varepsilon(r) \, dr = \frac{(s+\varepsilon)^{p}}{p}-s\varepsilon^{p-1}.
  \end{equation}
  We need to get rid of the parameter $\delta>0$, and for that, we integrate again in $t$ for $t \in (\tau, \tau+h)$,
  \[\int_{\mathbb{R}^N} F_\varepsilon(u^\delta(\tau+h))-F_\varepsilon(u^\delta(\tau)) +
    \frac{1}{\delta}\int_\tau^{\tau+h}\int_t^{t+\delta}\mathcal{E} \Big(\varphi(u(s)), f_\varepsilon \left( u(s) \right) \Big) \, ds dt = 0.\]
  Now we take $\delta \to 0$,
  \begin{equation} \label{extincion_casifinal}
    \int_{\mathbb{R}^N} F_\varepsilon(u(\tau+h))-F_\varepsilon(u(\tau)) +
    \int_\tau^{\tau+h}\mathcal{E} \Big(\varphi(u(t)), f_\varepsilon \left( u(t) \right) \Big)\, dt = 0.
  \end{equation}
  For the second term we use the Stroock-Varopoulos' inequality,
  \begin{equation} \label{SV}
    \mathcal{E}(F(u),G(u))\geq \overline{\mathcal{E}}(H(u)) \quad \text{whenever } (H')^2 \leq F'G',
  \end{equation}
  see \cite{StroockInequality}, \cite{VaropoulosInequality} and \cite{dePabloBrandleHeatEquations}.
  Take $F:=f_\varepsilon$ and $G:=\varphi$, and
  \[H:=h_\varepsilon(s) = \sqrt{c(p-1)} \int_0^s (r+\varepsilon)^{\frac{p-2}{2}} r^{\frac{m-1}{2} } \, dr \leq  \int_0^s \sqrt{f'_\varepsilon(r)} \sqrt{\varphi'(r)} \, dr. \]
  Then, with \eqref{SV} and the Sobolev embedding,
  \[\mathcal{E} \Big(\varphi(u(t)), f_\varepsilon \left( u(t) \right) \Big) \geq \overline{\mathcal{E}}(h_\varepsilon(t)) \geq C \left(\int_{\RR^N} h_\varepsilon(t)^{\frac{2N}{N-\sigma}}\right)^{\frac{N-\sigma}{N} }. \]
  Introducing this in \eqref{extincion_casifinal} we obtain
  \[ \int_{\mathbb{R}^N} F_\varepsilon(u(\tau+h))-F_\varepsilon(u(\tau)) + C
    \int_\tau^{\tau+h} \left(\int_{\RR^N} h_\varepsilon(t)^{\frac{2N}{N-\sigma}}\right)^{\frac{N-\sigma}{N} } \, dt \leq 0.\]
  Take $\varepsilon \to 0$ and, thanks to our choice of $p$, we get expression \eqref{discreto_final}.
\end{proof}

\begin{Remark}
  An interesting consequence of dividing the initial condition into $u_0^+$ and $u_0^-$ and proving extinction for both initial datum separately is that they may have different extinction times. Therefore, there might be a certain time $T_1$ for which our solution becomes entirely positive or negative, and another $T_2 \geq T_1$ for which the solution goes extinct.
\end{Remark}

\section*{Acknowledgements}

This research has been supported by grants CEX2019-000904-S, PID2020-116949GB-I00, and RED2022-134784-T, all of them funded by MCIN/AEI/10.13039/501100011033. \textsc{F.\! Quirós} and \textsc{J.\! Ruiz-Cases} were also supported  by the Madrid Government (Comunidad de Madrid – Spain) under the multiannual Agreement with UAM in the line for the Excellence of the University Research Staff in the context of the V PRICIT (Regional Programme of Research and Technological Innovation).



\bibliographystyle{acm}

\end{document}